\documentclass[11pt,a4paper]{article}

\usepackage[margin=1in]{geometry}

\usepackage[utf8]{inputenc}
\usepackage{amsmath}
\usepackage{amsfonts}
\usepackage{amssymb}
\usepackage{graphicx}
\usepackage{color}
\usepackage{amsthm}
\usepackage{comment}
\usepackage{enumerate}
\usepackage[show]{ed}
\usepackage{cleveref}
\usepackage[ruled,vlined]{algorithm2e}
\usepackage{multirow}

\newtheorem{theorem}{Theorem}[section]
\newtheorem{corollary}[theorem]{Corollary}
\newtheorem{lemma}[theorem]{Lemma}
\newtheorem{proposition}[theorem]{Proposition}
\theoremstyle{definition}
\newtheorem{definition}[theorem]{Definition}
\newtheorem{remark}[theorem]{Remark}
\newtheorem{example}[theorem]{Example}
\newtheorem{assumption}[theorem]{Assumption} 
\newtheorem{notation}[theorem]{Notation}

\numberwithin{equation}{section}

\newcommand{\C}{\mathrm{C}}
\newcommand{\quasiunif}{\rho} 
\newcommand{\phistab}{\Phi^{h}}
\newcommand{\Qstab}{Q_{h}}
\newcommand{\Qstabhat}{\widehat{Q}_{h}}
\newcommand{\Qhat}{\widehat{Q}_{h}}
\newcommand{\sgmin}{\sigma_{\min}}
\newcommand{\sgmax}{\sigma_{\max}}
\newcommand{\sgSmax}{(\sigma_S)_{\max}}
\newcommand{\sgSmin}{(\sigma_S)_{\min}}

\newcommand{\cS}{\mathcal{S}}
\newcommand{\cC}{\mathcal{C}}
\newcommand{\cK}{\mathcal{K}}
\newcommand{\cP}{\mathbb{P}}
\newcommand{\cSmu}{\mathcal{S}_\mu}

\newcommand{\cSmuk}{\mathcal{S}_{\mu_k}}
\newcommand{\Q}{\mathcal{Q}}

\newcommand{\uhat}{\widehat{u}}
\newcommand{\cE}{\mathbb{E}}
\newcommand{\V}{\mathbb{V}}
\newcommand{\cO}{\mathcal{O}}
\newcommand{\cLLtwo}{ {L_2} \mapsto {L_2} } 
\newcommand{\cLLinf}{ {L_\infty} \mapsto {L_\infty} }
\newcommand{\cLc}{\C \mapsto \C}
\newcommand{\LinftoC}{L_\infty \mapsto \C} 
\newcommand{\sgn}{\text{sgn}}
\newcommand{\rd}{\text{ d}	}
\newcommand{\wsigma}{\mathcal{P}^h \sigma } 
\newcommand{\cPh}{\mathcal{P}^{h}}

\newcommand{\cR}{\mathcal{R}(\sigma,\sigma_S)}
\newcommand{\cRdash}{\mathcal{R}'(\sigma,\sigma_S)}
\newcommand{\cRone}{\mathcal{R}_1(\sigma,\sigma_S)}
\newcommand{\cRtwo}{\mathcal{R}_2(\sigma,\sigma_S)}
\newcommand{\cRthr}{\mathcal{R}_3(\sigma,\sigma_S)}
\newcommand{\cRfr}{\mathcal{R}_4(\sigma,\sigma_S)}

\newcommand{\MC}{{MC}}
\newcommand{\Ein}{ \text{Ein} }

\newcommand{\quasunif}{\rho}
\newcommand{\brov}[1]{\overline{#1}}
\newcommand{\brun}[1]{\underline{#1}}
\newcommand{\hmax}{h^{\max}_{\omega}}

\newcommand{\igg}[1]{\textcolor{black}{#1}}
\newcommand{\iggb}[1]{\textcolor{black}{#1}}
\renewcommand{\aleph}{J}

\title{Full Error Analysis and Uncertainty Quantification for the Heterogeneous Transport Equation in Slab Geometry}

\author{Ivan G. Graham$^1$, 
  Matthew J. Parkinson$^1$ and
  Robert Scheichl$^{1,2}$%
}

\begin{document}

\maketitle
\begin{center}
\begin{scriptsize}

\vspace{-0.5cm}

\noindent
${}^1$ Dept of Mathematical Sciences, University of Bath, Bath BA2
7AY, UK.\\ \indent {\tt I.G.Graham@bath.ac.uk}

\vspace{0.1cm}

\noindent
${}^2$ Institut f\"ur Angewandte Mathematik, Universit\"at Heidelberg,
69120 Heidelberg, Germany. \\ \indent {\tt r.scheichl@uni-heidelberg.de}

\vspace{0.5cm}

\end{scriptsize}
\end{center}

\begin{abstract}
We present an  analysis of multilevel Monte Carlo techniques for  the forward problem of
    uncertainty quantification for the radiative transport equation, when the coefficients
    ({\em cross-sections})  are heterogenous random fields.  To do this,  we first give
    a new error analysis for the  combined spatial and angular discretisation in the deterministic case, 
    with  error estimates which are explicit in the  coefficients (and allow for very low regularity and jumps).
  This detailed error analysis is done for the 1D space  - 1D  angle slab geometry case with classical diamond differencing.    Under reasonable  assumptions on the statistics of the coefficients,  we then prove an  error estimate for the random
    problem in a suitable Bochner space. Because the problem is not self-adjoint, stability
    can only be proved under a path-dependent mesh resolution condition. This means that, while the Bochner space
    error estimate is of order $\cO(h^\eta)$ for some $\eta$, where   $h$ is a  (deterministically chosen)  mesh diameter,
    smaller mesh sizes might be needed for some realisations. We also show that the 
    expected cost for computing a typical quantity of interest remains of the same order as for a single sample.  
    This leads to rigorous complexity estimates for Monte Carlo and multilevel Monte Carlo:
    For particular linear solvers, the multilevel version gives up to two orders
    of magnitude improvement over Monte Carlo. We provide numerical results  supporting the theory.
\end{abstract}



\section{Introduction}
\label{sec:intro}

The Radiative Transport Equation (RTE) is a physically derived balance equation which models the angular flux $\psi$ of rarefied particles (such as photons or  neutrons) in a domain. Generally $\psi$ is a function of position, direction of travel, energy  and time (see e.g.  \cite{AdLa:02, BeGl:70, LeMi:84}). It is assumed that the particles cannot interact with each other and that they travel along straight line paths with some energy
until they interact with larger nuclei via absorption, scattering or fission.
The rates $\sigma_A$, $\sigma_S$ and $\sigma_F$ at which these collision events occur are called the \textit{absorption, scattering and fission cross-sections}. The RTE  has many applications, for example in radiation shielding, nuclear reactor design \cite{BeGl:70, SaMc:82},
astrophysics and optical tomography \cite{DuTuCh:10, Re:10}.

In the context of neutron transport, the two main scenarios of interest are the so-called \textit{fixed source problem} and the \textit{criticality problem}. We focus on the former, which concerns the transport and scattering of particles emanating from some fixed source  $f$.
In steady state, with constant energy, assuming isotropic scattering  and neglecting fission, the problem can be written as the integro-differential equation: 
\begin{align}
  \label{eq:fullnte}
\left[ \vec{\Theta} \cdot \nabla + \sigma(\vec{r}) \right] \psi(\vec{r},\vec{\Theta}) \ = \  
 \sigma_S(\vec{r}) \phi(\vec{r})  \ + \  f(\vec{r},\vec{\Theta})  \end{align}
with independent variables being angle   $\vec{\Theta} \in \mathbb{S}^2$ (the unit sphere in $3D$),  and position $\vec{r} \in V$ (the physical domain occupied by the reactor),  \igg{where $\psi(\vec{r},\vec{\Theta})$ is the angular flux,  and}   
\begin{align} \label{eq:scalar_flux} 
\phi(\vec{r}) \ :=\ \frac{1}{4 \pi} \int_{\mathbb{S}^2} \psi(\vec{r}, \vec{\Theta}) \rd \vec{\Theta}
  \end{align} 
  is the {\em scalar flux}.  Equation \eqref{eq:fullnte} requires boundary conditions and here we will restrict to the zero incoming flux condition
  \begin{align} \label{eq:fullbc}
    \psi(\vec{r}, \vec{\Theta}) = 0, \quad \text{when} \quad \vec{r} \in \partial V  \quad \text{and} \quad  \vec{n}(\vec{r}) \cdot \vec{\Theta} < 0, 
  \end{align}
  with $\vec{n}(\vec{r})$ denoting the outward  normal from $V$ at a point $\vec{r} \in \partial V$.  
The gradient $\nabla$ is with respect to $\vec{r}$ and the coefficient function $\sigma(\vec{r})$ is the \textit{total cross-section} defined by
\begin{equation} \label{eq:total_xsection}
\sigma(\vec{r}) \ = \ \sigma_A(\vec{r}) +  \sigma_S(\vec{r}) \ ,
\end{equation}
where $\sigma_S$ and $\sigma_A$ are, respectively, the scattering and
absorption cross-sections, both assumed to be non-negative.

In this paper we propose and analyse  efficient multilevel Monte Carlo methods for quantifying the effect of uncertainty in the {\em input data} $\sigma_A$, $\sigma_S$ and  $f$,  on the {\em output variable} $\phi$.   There is a growing recent interest in this question in the more general  context of kinetic equations.   
In the particular case of  nuclear applications,  our work is relevant to the assessment of how  material fluctuations  can affect the   
uncertainty of flux computations.

\medskip

\noindent {\bf Novel results in this paper.}
Since their introduction in the context of high-dimensional quadrature and SDEs in mathematical finance \cite{He:01, Gi:08},  multilevel Monte Carlo methods have generated a lot of interest. While uncertainty quantification recently
  has become a topic of  great general  interest for 
the transport equation in both theory and practice
e.g. \cite{JiLuPa:18, ZhLi:18}, to our knowledge \igg{the present}  work provides
the first rigorous analysis of multilevel Monte Carlo methods for this
problem. To allow the first results to be established, we make the simplifying assumption  of
one spatial and one angular dimension, the so-called ``slab-geometry'' case in reactor theory. We discretise with the classical 
discrete ordinates method, using a certain Gauss rule with $2N$ quadrature points in angle and classical diamond
differencing (or Crank-Nicolson) on a mesh with step-size $h$ in the spatial variable.
The resulting approximation of the scalar flux $\phi$  is denoted $\phi^{h,N}$ .

 Our first set of results describes how heterogeneity in the material coefficients manifests itself
in  the  operators underlying the RTE,   and consequently in the error estimate for the numerical method.
We  assume  that  the spatial domain can be partitioned into subintervals, on each of  which 
the input data $\sigma_S$, $\sigma_A$ and $f$ belongs to the H\"{o}lder space  $\C^\eta$,
for some {$\eta \in (0,1)$}. This allows for data with  low smoothness and permits jumps in material properties across interfaces. We denote this space by $\C_{pw}^\eta$ and equip it with the norm $\| \cdot \|_{\eta,pw}$ defined below.
Our first error estimate   is Theorem \ref{thm:phiMN}, which shows that there are  constants $\mathcal{R}, \mathcal{R}'$, both dependent on
$\sigma, \sigma_S$ such that, when
\begin{align}\label{eq:path_dep} 
   N^{-1} \ + \ h \log N \ + \ {h^{\eta}} \ \leq \ {\cR^{-1}},  
  \end{align}
 we have the error estimate for the scalar flux: 
\begin{equation} \label{eq:result1}
\| \phi - \phi^{h,N} \|_\infty \ \leq \ \mathcal{R}'(\sigma, \sigma_S) \left( N^{-1} \ + \ h \log N \ + \ h^{\eta} \right) \| f \|_{\eta,pw} \ .
\end{equation}
Both  $\mathcal{R}, \mathcal{R}'$ are  independent of $f$ and  their  dependence on $\sigma, \sigma_S$ is known explicitly, {but they blow up, e.g., if $\sigma/\sigma_S $ is close to $1$ anywhere} in the domain or if $\| \sigma \|_\infty\|\sigma^{-1} \|_\infty \to \infty$.
The proof of \eqref{eq:path_dep}, \eqref{eq:result1} is obtained by generalising the theory of the  integral equation reformulation of
\eqref{eq:fullnte} to the heterogeneous case; the  homogeneous case having been studied in detail in   \cite{PiSc:83}. An overview \igg{of the present work} was given in \cite{GrPaSc:17}.

The appearance of the  $h \log N$ term in \eqref{eq:result1}  reflects the fact that the transport equation in slab geometry has a singularity in its angular dependence (explained in \S \ref{sec:ModelProblem}). This imposes a compatibility constraint, which implies that the angular discretisation cannot be indefinitely refined if the  spatial discretisation is kept fixed. The appearance of this term in the
error estimate means that the accuracy of the method measured in $\Vert \cdot \Vert_\infty$ can be no better than $\mathcal{O}(h)$,  even if the cross-sections are very smooth. A faster rate is possible if one uses a higher order method or measures the error in $L_p$ norms, the latter proved for constant cross-sections in  \cite{PiSc:83}. However, we will not pursue this further and thus limit our analysis to piecewise H\"{o}lder continuous data ($\eta < 1$).
To ensure that the spatial and angular errors
are  equal order,  we set $N = N(h) =   \lceil h^{-\eta}\rceil $.

Our second set of results then  concerns the probabilistic counterpart of \eqref{eq:result1}.
Here we have to deal with the fact that the deterministic
estimate \eqref{eq:result1} is subject to the   ``mesh resolution condition'' \eqref{eq:path_dep}, which in turn arises from   
the  non-self-adjointness of  \eqref{eq:fullnte}.
In the case of coercive  self-adjoint PDEs with random data and Galerkin discretisation (e.g. \cite{Cl:11,  TeScGiUl:13}) one obtains a
probabilistic error estimate by interpreting  the deterministic error estimate  pathwise  and then taking expectation.
This does not work here because of the pathwise stability estimate  \eqref{eq:path_dep}.
To get around this problem, given a path independent mesh width $h<1$,
for each realisation $\sigma = \sigma(\cdot, \omega), \sigma_S =
\sigma_S(\cdot, \omega)$, we let $\hmax$ denote (the largest)  mesh
diameter which satisfies the path-dependent  criterion
\eqref{eq:path_dep} and then set $h_\omega = \min\{ h, \hmax \}$. Then the approximation to $\phi = \phi(\cdot, \omega)$ is taken to be $\Phi^h = \phi^{h_\omega, N(h_\omega)}$. We prove in Theorem \ref{thm:phiMNrand} that  
\begin{equation} \label{eq:result2}
\| \phi - \Phi^{h} \|_{L_p(\Omega;L_\infty)} \ \leq \ C_{p,r} \,  h^\eta \, \| f \|_{L_r(\Omega;\C_{pw}^\eta)} \ ,
\end{equation}
for any $1 \leq p \leq r$, provided the norm on the right-hand side is finite and the cross sections $\sigma$, $\sigma_S$ have bounded moments of any finite order. Here,  $C_{p,r}$ denotes an  absolute constant depending only on $p, r$ and the norms are the usual Bochner norms with respect to the probability space $\Omega$ (defined in \Cref{sec:numerics}). 
This result shows that the error in the Bochner norm on the left-hand
side decreases with deterministic rate $h^\eta$, provided we are
willing to use a finer mesh for any particular sample where the
stability criterion \eqref{eq:path_dep} demands it. If we assume furthermore that the cost $\cC(\cdot)$ to compute a single sample of $\Phi^h = \phi^{h_\omega, N(h_\omega)}$ (e.g. measured in floating point operations) satisfies
$$
\mathcal{C}(\phi^{h_\omega,N(h_\omega)}) \ \leq \ C'(\omega) h_\omega^{-\gamma} \ ,
$$
for some $\gamma > 0$, and that the sample-dependent constant $C'$ in that estimate is in $L_p(\Omega)$, for some $p>1$, then the third main result of this paper in \Cref{lem:expect_cost} is that
\begin{equation} \label{eq:result3}
\cE[ \mathcal{C}(\phistab) ] \ = \ \mathcal{O}\left( h^{-\gamma} \right) \ ,
\end{equation}
where the hidden constant is independent of $h$. The important observation is that on average the cost to compute a sample from $\phistab$ has the same cost growth rate (w.r.t. $h$) as the sample-wise cost (w.r.t. $h_\omega$), despite some samples $\phistab(\omega,x)$ being computed on a mesh with $h_\omega \ll h$ in order to satisfy the stability criterion.

Estimates, such as \eqref{eq:result2} and \eqref{eq:result3}, play a crucial role  in the complexity analysis of   (multilevel) Monte Carlo methods
for computing the expectation of (functionals of) the solution $\phi$ of \eqref{eq:fullnte}.  Suppose 
 $Q(\phi)$ is such a functional (often called a \textit{quantity of interest}) and to simplify notation we write this as
$Q$  (a random variable).
We approximate  $Q$ by  $Q_h : = Q(\Phi^h)$ with $\Phi^h$ described  above and then approximate   $\cE[Q]$, by applying a sampling method of choice to $\mathbb{E}[Q_h]$ -- we denote  the result as  $\widehat{Q}_h$. Finding an accurate and efficient estimator $\widehat{Q}_h$ of $\cE[Q]$ is at the heart of the forward problem of Uncertainty Quantification (UQ).

To compare methods in UQ, the \textit{computational $\epsilon$-cost} $\cC_{\epsilon}(\widehat{Q}_h)$ of an estimator $\widehat{Q}_h$ is often considered. If  $\epsilon$ denotes a desired accuracy (in the sense of root mean-squared error),
then $\cC_{\epsilon}(\widehat{Q}_h)$ is defined to be the total cost for $\widehat{Q}_h$
to achieve an accuracy of $\epsilon$. By a general theory
in~\cite{Cl:11}, the  $\epsilon$-cost of standard and multilevel Monte
Carlo methods can be computed in terms of the parameter $\eta$ in
\eqref{eq:result2} (related to the regularity of the data), the
parameter $\gamma$ in \eqref{eq:result3} (related to the cost per sample), as well as another parameter $\beta$ that quantifies the speed of variance reduction between levels of the multilevel scheme and can also be derived from \eqref{eq:result2}.  In the fourth main result of this paper in \Cref{prop:epsMC}, we prove rigorously that
\begin{equation} \label{eq:result4}
\beta \; \ge \; 2 \eta \ .
\end{equation}

To provide a bound on $\gamma$ that only depends on the regularity of the data  it is necessary to fix the solution method. Two particular examples that were used in our numerical results in \cite{GrPaSc:17} are given in \Cref{ex:solvers}. In particular, for the asymptotically cheaper one of the two methods, which is an iterative procedure called source iteration, we have $\gamma \le  1+\eta$. The general theory in \cite{Cl:11} then leads to the following respective upper bounds on the $\epsilon$-costs of the standard and the multilevel Monte Carlo estimators $\widehat{Q}_{h}^{MC}$ and $\widehat{Q}_{h}^{MLMC}$:
$$
\cE\Big[\cC_{\epsilon}(\widehat{Q}_{h}^{MC})\Big] \ = \ \cO \big( \epsilon^{-\left(4 + \frac{1-\eta}{\eta} \right)} \big)  \ \ \ \text{ and } \ \ \ \cE\Big[\cC_{\epsilon}(\widehat{Q}_{h}^{MLMC})\Big] \ = \ \cO \big( \epsilon^{- \left(2 + \frac{1-\eta}{\eta} \right)} \big)\;,
$$
i.e., a theoretical gain of up to two orders of magnitude in $\epsilon^{-1}$. However, we will see in the numerical section that this estimate of improvement is overly optimistic, since the bound on the $\epsilon$-cost of the standard Monte Carlo estimator is not sharp. Nevertheless, we do observe gains of (at least) one order of magnitude in practice.

\medskip

\noindent 
{\bf Related literature.} \ \
The numerical analysis of the RTE (and related integro-differential equation problems) dates back at least as far as the work of H.B. Keller \cite{Ke:60}. After a huge growth in the  mathematics   literature in the 1970's and 1980's,  progress has been slower  
since. This is perhaps surprising,  since  discontinuous Galerkin (DG) methods have enjoyed a massive recent renaissance and the neutron transport problem was one 
of the motivations behind the original introduction of DG \cite{ReHi:73}.  

The fundamental paper on the  analysis of the discrete ordinates
method  for the transport equation is
\cite{PiSc:83}, where a  full analysis of the combined effect of angular and spatial discretisation is given  under the assumption that the cross-sections are  constant. 
The delicate relation between spatial and angular discretisation parameters
required to  achieve stability and convergence is described there, and is also seen again in the present work (see the $h \log N$ term in \eqref{eq:result1}). 
Later research e.g. \igg{\cite{AsKuLa:92}}, \cite{As:98}, \cite{As:09}  
produced analogous results for models of increasing complexity  and 
in higher dimensions, but the proofs  were mostly confined to the case of cross-sections that are constant in space. 
A separate and related sequence of papers  (e.g. \cite{LaNe:82}, \cite{Vi:84},  and  \cite{AlViGa:89})
allow for variation in cross-sections, but error estimates explicit in this data  are not available there.
\igg{A method for tackling directly the integral equation reformulation of the RTE is given in \cite{FuRa:96}. Again the analysis is not explicit in the heterogeneity.   }

\igg{While  the  coefficient-explicit analysis given here can in principle be extended to higher spatial and angular dimensions (a  start is contained  in \cite{BlGrScSp:18}), it is clear that
  the details will be quite formidable, so we restrict here to the 1D case. However we note that  there is substantial contemporary interest in practical 3D  modelling in the heterogeneous case (generally without rigorous analysis). For example,  the recent thesis \cite{Gil:18} solves the  multigroup in energy approximation of the $P_N$ angular approximation of the transport equation  using a non-conforming
  spatial discretization for a highly heterogeneous reactor,  using  a domain decomposition approach.}

The field of UQ has grown very quickly in recent years and its application to radiative transport theory is currently of considerable interest. There are a number of groups that already work on UQ in radiative and neutron transport, e.g.  \cite{AyEa:15, Fi:11, Gi6:13} and references therein. 
The most recent research has focussed on using the polynomial chaos expansion (PCE), combined with a collocation method to estimate the coefficients in the expansion. The main disadvantage of standard PCE is that the number of terms grow exponentially in the number of stochastic dimensions and in the order of the PCE, the so-called \textit{curse of dimensionality}. A variety of techniques have been used to remedy this, including (adaptive) sparse grids  \cite{Gi6:13}, hybrid mixtures of polynomials \cite{AyPaEa:14} and expanding the quantity of interest in terms of low-dimensional subspaces of the stochastic variable \cite{AyEa:15}. We note however that none of these papers provide any rigorous error or cost analysis.

We also  note that there is a growing literature in the numerical analysis of  kinetic equations, of which the RTE is a particular example, with an emphasis on asymptotic preserving schemes which retain accuracy as the scattering ratio $\sigma_S/\sigma$ approaches unity. Interest in this question in the deterministic case goes back a long way, e.g. \cite{JiLi:91}, which has led to recent work on UQ in this context (e.g. \cite{ZhLiJi:16}).   Recently modern operator compression and adaptive techniques have been applied to efficiently attack  the high-dimensional aspects of the transport problem \cite{DaGrMu:18}.    
For further general discussion on the transport equation, see \cite{DaLi:12, LeMi:84}.
 
By contrast our work focusses on multilevel Monte Carlo and  sampling methods \cite{GrPaSc:17}.
Monte Carlo is inherently dimension independent and Quasi-Monte Carlo can be proved to be so under certain conditions, e.g. \cite{GrKu:15}. As far as we know, these methods have not been applied to radiative transport until now. Our  previous paper \cite{GrPaSc:17} gave an overview of this topic and also investigated the multilevel quasi-Monte Carlo method for the RTE. Further details,
are in
\cite{Pa:18}.

\medskip

\noindent 
{\bf Structure of paper and notation.} \ \ 
In Section 2, we introduce the model problem; the Radiative Transport equation in slab geometry with spatially heterogeneous cross-sections and its discretisation. To set up the error analysis, Section 3 describes the classical integral equation reformulation of the RTE under very weak smoothness assumptions on the cross-sections. From here we can prove results relating to the underlying operators and their regularity - that are explicit in the cross-sections. In Section 4, the elements are brought together to prove \eqref{eq:result1}.  We introduce uncertainty into the input data in Section 5, and we  extend the error estimate of Section 4 to the probabilistic error estimate \eqref{eq:result2} and subsequently prove \eqref{eq:result3}. Numerical results are given, with the cross-sections assumed to be log-normal random fields equipped with the Mat{\'e}rn class of covariances, and represented by a Karhunen-Lo{\`e}ve expansion.   An overview of the results of this paper, without  detailed analysis  was previously presented in \cite{GrPaSc:17}.  

\section{The Model Problem}
\label{sec:ModelProblem}

We study the \textit{mono-energetic 1D slab geometry problem}, for the angular flux $\psi(x,\mu)$: 
\begin{align} \label{eq:transport_det} 
\mu \frac{\partial \psi}{\partial x}(x,\mu) \ + \ \sigma(x) \psi(x,\mu)  &\ = \ \sigma_S(x) \phi(x) \ +  \ f(x) \ , \quad x \in (0,1), \quad \mu \in [-1,1] , \\
\label{eq:phi_det}
\text{where} \qquad \phi(x) &\ = \ \frac{1}{2} \int_{-1}^1 \psi(x,\mu') \ d \mu'
\end{align}
denotes the scalar flux, subject to zero incoming flux: 
\begin{equation}\label{eq:bc} 
\psi(0,\mu) \ = \ 0, \ \text{ for } \ \mu > 0 \quad \text{and} \quad \psi(1, \mu) = 0, \ \text{ for } \ \mu < 0 \ .
\end{equation}
The total cross-section $\sigma(x)$ is given by $\sigma = \sigma_S + \sigma_A$.
The problem  \eqref{eq:transport_det} --  \eqref{eq:bc} is obtained  from
\eqref{eq:fullnte} -- \eqref{eq:fullbc} when the
input data  is constant in two of the spatial dimensions (here
assumed to be $y$ and $z$).
Note that \eqref{eq:transport_det} degenerates at $\mu = 0$, which corresponds to   particles moving perpendicular to the $x$-direction.
 
\begin{notation}
When working on  the spatial domain $(0,1)$,  for $1 \le p \le
\infty$, we will denote the standard Lebesgue spaces
as $L_p$ with norm $\Vert \cdot \Vert_p$. 
  For any interval $I \subset [0,1]$, we denote by $\C(I)$ the space of uniformly continuous functions on $I$, equipped with norm $\|\cdot\|_\infty$. Any function $g \in \C(I)$ has a unique continuous
  extension to $\overline{I}$. For $0 < \xi \le 1$, we let $\C^\xi(I)$ denote the space of H\"older continuous functions on $I$ with H\"older exponent $\xi \in (0,1]$ and with norm $$\|g\|_{\C^\xi(I)} := \|g\|_\infty + \sup_{x,y \in I} \frac{|g(x) - g(y)|}{|x-y|^\xi}.$$
When $I=[0,1]$, we write for short $\C=\C(I)$, $\C^\xi = \C^\xi(I)$ and $\|f\|_{\xi}  = \|f\|_{\C^\xi(I)}$.
Finally,  for any normed spaces $X$ and $Y$, we write $\| . \|_{X \mapsto Y}$ to denote the operator norm of an operator mapping $X \mapsto Y$.
\end{notation}

In what follows, we will allow data which is piecewise  continuous with respect to an a priori defined  partition
  \begin{align} \label{eq:part} 0 = c_1 < ... < c_{\aleph} = 1, \end{align}
with $\aleph \geq 2$. 
  We denote the corresponding space of piecewise continuous functions by
$$
\C_{pw} := \big\{ g \in L_\infty[0,1]:  g|_{(c_j, c_{j+1})} \in C(c_j, c_{j+1}),  \ \text{for each} \ j=1,\ldots,\aleph-1 \big\} .
$$
For definiteness we will assume
that the value of $g(c_j)$ is taken to be the limit from the right for $j = 1, \ldots,  \aleph -1$ and the limit from the left for $j = \aleph$.  The space $C_{pw}$ is  equipped with the usual uniform norm $\Vert \cdot \Vert_{\infty}$.  
 Similarly, for any $\xi \in (0,1]$, let
$$
\C^\xi_{pw} := \big\{ g\in \C_{pw}: g|_{(c_j, c_{j+1})} \in \C^\xi(c_j,c_{j+1}), \ \text{for each} \ j=1,\ldots, \aleph-1 \big\}
$$
with norm $\|g\|_{\xi,pw} := \max_{j=1}^{\aleph} \| g \|_{\C^\xi(c_j,c_{j+1})}$.

We now make the following physically motivated assumptions on the data. 

\begin{assumption}{(Input Data)}
\label{ass:cross}
\begin{enumerate}
\item The cross-sections $\sigma_S$ and $\sigma_A$ are strictly positive and bounded above. We write 
$$\sgmin = \min_{x\in[0,1]} \sigma(x),  \ \ \sgmax = \max_{x\in[0,1]} \sigma(x),  \ \ (\sigma_S)_{\min} = \min_{x\in[0,1]} \sigma_S(x)  \ \text{and} \  \ (\sigma_S)_{\max} = \max_{x\in[0,1]} \sigma_S(x).$$
\item There exists a partition \eqref{eq:part} and {$\eta \in (0,1]$}, such that  $\sigma, \sigma_S, f \in \C^\eta_{pw}$.
\end{enumerate}
\end{assumption}

\medskip

\subsection{Discretisation}
\label{sec:quadrule}

To discretise \eqref{eq:transport_det} -- \eqref{eq:bc} in angle, we use a  $2N$-point quadrature rule 
\begin{equation}
\label{eq:quad_generic}
\int_{-1}^1 g(\mu) d\mu \approx \sum_{|k|=1}^N w_k g(\mu_k)\,,
\end{equation}
with nodes $\mu_k \in [-1,1]\backslash\{0\}$ and positive weights $w_k \in \mathbb{R}$. We assume the (anti-) symmetry properties $\mu_{-k} = -\mu_k$ and $w_{-k} = w_k$. To discretise in space, we introduce a mesh  
\begin{equation}
\label{eq:mesh}
0 = x_0 < x_1 < \ldots < x_M = 1,   
\end{equation}
which is assumed to  resolve the break points $\{ c_j\}$ introduced in \eqref{eq:part}. \igg{We set
  $ h_j = x_j - x_{j-1}$.}  Further assumptions on the quadrature rule and mesh will be added in Section \ref{sec:fulldiscrete}.

Our discrete scheme for \eqref{eq:transport_det} -- \eqref{eq:bc} is then
\begin{equation} \label{eq:fulldisc}
\mu_k  \frac{ \psi_{k,j}^{h,N} - \psi_{k,j-1}^{h,N}}{h_j} \ + \ \sigma_{j-1/2} \frac{\psi_{k,j}^{h,N} + \psi_{k,j-1}^{h,N}}{2} \ = \ \sigma_{S,j-1/2} \phi_{j-1/2}^{h,N} \ + \ f_{j-1/2}   \ ,
\end{equation}
for $j = 1,...,M, \ |k| = 1,\ldots,N$, where 
\begin{equation}
\label{eq:phi_fulldisc}
\phi_{j-1/2}^{h,N}\ = \ \frac{1}{2} \sum_{|k| = 1}^N w_k \frac{\psi_{k,j}^{h,N} \ + \ \psi_{k,j-1}^{h,N}}{2} \ , \ \ j = 1,...,M \ ,
\end{equation}
and with
\begin{align} \label{eq:disc_bc} \psi_{k,0}^{h,N} = 0, \quad \text{for} \quad  k >0 \quad \text{and} \quad  \psi_{k,M}^{h,N} = 0, \quad
  \text{for} \quad k <0\ .  \end{align}
Here $\sigma_{j-1/2}$ denotes the value of $\sigma$ at the mid-point of the interval $I_j = (x_{j-1}, x_j)$, with the analogous  meaning for   
$\sigma_{S,j-1/2}$ and $f_{j-1/2}$. 

\subsection{Abstract form  of {the} method} \label{sec:opform}

As preparation for analysing \eqref{eq:transport_det} -- \eqref{eq:bc} and its discretisation, 
\eqref{eq:fulldisc} -- \eqref{eq:disc_bc}, consider first the \textit{pure transport problem}: For fixed $\mu \in [-1,1]$, find $u = u(x), \, x \in (0,1)$,  such that
\begin{equation} \label{eq:ptr}  \mu \frac{d u}{d x} + \sigma u = g, \quad  \text{with} \ u(0) = 0, \ \text{when} \ \mu > 0 \quad  \text{and}\  u(1) = 0 \ \text{when}\  \mu < 0,  \end{equation} 
with $g \in L_\infty$  a generic right-hand side. (Note that $u$
depends on $\mu$, but we suppress this in the notation. When $\mu =0$ no boundary condition is needed.) It is easy to show that the
unique solution of this problem is $u := \cSmu g$, where 
\begin{align}
 \cS_\mu g(x)  \ &= \ \left\{ \begin{array}{lr} \ \ \mu^{-1} \int_0^x \exp\left(\mu^{-1} \tau(x,y) \right) g(y) \ d y \ , & \ \mu > 0 \\
 \ \ \sigma^{-1}(x) \ g(x) \ , & \ \mu = 0 \\
- \mu^{-1} \int_x^1 \exp\left(\mu^{-1} \tau(x,y) \right) g(y) \ d y \ , & \ \mu < 0 \end{array} \right. \ ,
\label{eq:solop}
\end{align} 
and 
\begin{equation}
\label{eq:tau}
\tau(x,y) \ := \ \int_x^y \sigma(s) \ d s \ . 
\end{equation}
The quantity $\vert \tau(x,y) \vert$ is often called the `optical length' or `optical path' \cite{BeGl:70}. 
To mimic  the averaging process in \eqref{eq:phi_det} it is natural to also consider the  integral operator:  
\begin{equation}
\label{eq:Kdef}
\cK g(x) \ := \ \frac{1}{2} \int_{-1}^1 \cS_\mu g(x) \ d \mu \ = \ \frac{1}{2} \int_0^1 E_1(\vert\tau(x,y)\vert ) g(y) \ d y \ ,
\end{equation}
where for $z > 0$, $E_1(z)$ is the exponential integral 
\begin{equation} \label{eq:exp_integral}
E_1(z) \ := \ \int_1^\infty \exp(-t z) \frac{\rd t}{t} \ = \ \int_0^1 \exp(-z/s) \frac{\rd s}{s} \ . 
\end{equation}
The operators $\cSmu$ and $\cK$ relate to \eqref{eq:transport_det} --  \eqref{eq:bc} by the following proposition.

\begin{proposition}
\label{prop:ie}
Let $\psi$ be {a} solution to \eqref{eq:transport_det} --  \eqref{eq:bc}. Then,
\begin{equation}
\label{eq:phiprop}
\psi(x,\mu) \ = \ \cSmu \left( \sigma_S \phi \ + \ f \right)(x) 
\end{equation}
and hence, $\phi$ solves the integral equation
\begin{equation}
\label{eq:ie}
\phi \ = \ \cK \left( \sigma_S \phi \ + \ f \right) \ .
\end{equation}
\end{proposition} 
We shall see later that \eqref{eq:ie} has a unique solution and this ensures that \eqref{eq:transport_det} --  \eqref{eq:bc}  has a unique solution.
Analogously we can consider the discrete system  \eqref{eq:fulldisc} -- \eqref{eq:disc_bc}. Let  $V^h$ denote the space of
continuous piecewise-linear functions with respect to the mesh $\{x_j\}_{j=0}^M$, and for any $v \in \C$, let   $\cPh v  $ denote the piecewise constant function which interpolates $v$ at the mid-points of subintervals.  Then consider the discretisation of \eqref{eq:ptr} defined by seeking  $u^h \in V^h$ to satisfy  
\begin{equation} \label{eq:ptr_approx}
\int_{I_j} \left(\mu \frac{\rd u^h }{\rd x } + \wsigma \ u^h \right)\ = \ \int_{I_j} g \  , \quad \text{with} \quad 
I_j = (x_{j-1}, x_j) \ , \quad j = 1, \ldots, M \  ,
\end{equation}
with $u^h(0) = 0$ when $\mu > 0$ and $u^h(1) = 0$ when $\mu < 0$. This has a unique solution, which we write as
$u^h = \cS_\mu^h g$. Analogously to \eqref{eq:Kdef} we also define
\begin{equation}
\label{eq:KMNdef}
\cK^{h,N} g \ = \ \frac{1}{2} \sum_{|k|=1}^N w_k \ \cS^h_{\mu_k} g \ . 
\end{equation}

Identifying  any fully discrete  solution  $\psi_{k,j}^{h,N}$ of \eqref{eq:fulldisc} -- \eqref{eq:disc_bc} with
the function $\psi_{k}^{h,N} \in  V^h$ by interpolation at the nodes $\{x_j\}$,
we can see that \eqref{eq:fulldisc} -- \eqref{eq:disc_bc}  is equivalent to seeking  $\psi_k^{h,N} \in V^h$,\  $|k| = 1,...,N$,  that satisfy  
\begin{equation} \label{eq:rte_approx}
\int_{I_j} \left(\mu_k \frac{\rd \psi_k^{h,N} }{\rd x } + \wsigma \ \psi_k^{h,N} \right)\ = \ \int_{I_j} \cPh \left( \sigma_S \phi^{h,N} + f \right) \  , \quad j = 1, \ldots, M \  ,
\end{equation}
where 
\begin{align}\label{eq:rte_approxphi}
\phi^{h,N} \ = \ \frac{1}{2} \sum_{|k|=1}^N w_k \psi_k^{h,N} 
\end{align}
and
\begin{align} \label{eq:rte_approxbc}  \psi_k^{h,N} (0) =0 \quad \text{when } \quad k > 0 \quad \text{and} \quad \psi_k^{h,N} (1) =0 \quad \text{when } \quad k < 0  . \end{align}

We then have the discrete analogue of Proposition \ref{prop:ie}:  
\begin{proposition} \label{prop:ie_disc}
  The  system \eqref{eq:rte_approx} -- \eqref{eq:rte_approxbc}  is equivalent to \eqref{eq:fulldisc} -- \eqref{eq:disc_bc}, and its solution can be written: 
\begin{equation} \label{eq:fd_psi}
\psi_{k}^{h,N} \ = \  \cS_{\mu_k}^h \cPh (\sigma_S \phi^{h,N} + f) \ , \quad \ |k| = 1,...,N \ .
\end{equation}
Moreover, 
\begin{equation}
\label{eq:fullie}
\phi^{h,N} \ = \ \cK^{h,N} \cPh (\sigma_S \phi^{h,N} \ + \ f) \ . 
\end{equation}
\end{proposition}

Now, to estimate the error in our approximation to $\phi$, we use  \eqref{eq:ie} and \eqref{eq:fullie} to obtain 
\begin{equation} 
\phi - \phi^{h,N} \ = \ ( I -\mathcal{K}^{h,N} \cPh \sigma_S )^{-1} (\mathcal{K} - \mathcal{K}^{h,N}\cPh) ( \sigma_S \phi + f) \ . \label{eq:IE3}
\end{equation} 
We prove   later in
\eqref{eq:pstab} that $(I - \cK^{h,N}\cP^h \sigma_S)^{-1}$ is bounded on $\C$. Hence,
\begin{equation} 
\Vert \phi - \phi^{h,N}\Vert_\infty \ \leq \ \Vert ( I -\mathcal{K}^{h,N} \cPh \sigma_S )^{-1}\Vert_{\C \mapsto \C}  \Vert (\cK - \cK^{h,N} \cPh) ( \sigma_S \phi + f)\Vert_\infty \ .  \label{eq:IE4} 
\end{equation}
In \S \ref{sec:fulldiscrete}, we use \eqref{eq:IE4} to obtain a
data-explicit error estimate for $\phi - \phi^{h,N}$. First, we  prove
a number of  data-explicit properties of the operator $\cK$ which will be needed later. 

\section{Properties of the Operators}
\label{sec:analytic}

\begin{notation}\label{not:bar} 
To simplify presentation, for any $a \in \mathbb{R}$, we will use the notation $\brov{a} := \max \lbrace 1, a \rbrace$ and
    $\brun{a} : = \min\{a,1\}$. Also, from now on, we will use $c$ to denote a constant that is positive, finite and independent of the cross-sections, mesh parameters and other relevant variables.
\end{notation}

We will make use of the following bounds, a consequence of Assumptions \ref{ass:cross} and \eqref{eq:tau},
\begin{equation}
\label{eq:taubnd}
\sgmin | y - x | \leq \sgn(y - x) \tau(x,y) \ = \ \vert \tau(x,y) \vert \  \leq \sgmax | y - x | \ ,
\end{equation}
where $\text{sgn}(\cdot) = 1$, when its argument is positive, and $(-1)$ when negative.

\begin{lemma} 
\label{lem:Smustab}
For $\mu \in [-1,1]$, then
$\|  \cSmu \|_{\cLLinf} \leq \sgmin^{-1} \ .$
\end{lemma}

\begin{proof}
Consider $\mu > 0$ and a function $g \in L_\infty$. By definition,
\begin{align*}
| \cSmu g(x) | &= \bigg \lvert \mu^{-1} \int_0^x \exp( \mu^{-1} \tau(x,y) ) g(y) \ d y \bigg \rvert \leq \| g \|_{\infty} \mu^{-1} \int_0^x \exp( \mu^{-1} \tau(x,y) ) \ d y \\
&\leq \| g \|_{\infty} \mu^{-1} \int_0^x \exp( \mu^{-1} \sgmin (y - x) ) \ d y \\
&\leq \| g \|_{\infty} \sgmin^{-1} \left[ \exp( \mu^{-1} \sgmin (y - x) ) \right]_{y=0}^{y=x} \ \leq \ \| g \|_{\infty} \sgmin^{-1} \ ,
\end{align*}
where we have used \eqref{eq:taubnd}. The proof for $\mu < 0$ is similar.
\end{proof}

In the following two Lemmas we study the differentiability of $\cSmu$ with respect to $x$ and $\mu$. Through \eqref{eq:phiprop}, this relates directly to the differentiability of the angular flux $\psi$ and hence will be fundamental to the convergence rate of the deterministic error estimates in Section \ref{sec:fulldiscrete}.

\begin{lemma} 
\label{lem:SmuXderiv}
$$
\bigg \| \frac{\partial}{\partial x} \cSmu \bigg \|_{\cLLinf} \ \leq \ 2 \left( \frac{\sgmax}{\sgmin}\right) |\mu|^{-1} \ ,  \quad \text{ for all } \ \mu \in [-1,1] \backslash \{ 0 \} \ .
$$
Moreover, $\cSmu: L_\infty \mapsto \C$ with $\| \cSmu \|_{L_\infty \mapsto \C} \leq \sgmin^{-1}$, for $\mu \in [-1,1] \backslash \{ 0 \}$.
\end{lemma}

\begin{proof}
Consider $\mu > 0$ and some $g \in L_\infty$. Applying the Leibniz integral rule to \eqref{eq:solop}, 
\begin{align*}
  (\cSmu g)'(x)
&= \mu^{-1} \left( g(x)  - \mu^{-1} \sigma(x) \int_0^x \exp( \mu^{-1} \tau(x,y) ) g(y) \ d y  \right) 
\end{align*}
and thus 
\begin{align*}
 \lvert (\cSmu g)'(x) \rvert &\leq \mu^{-1} \left( | g(x) |  + \mu^{-1} \sgmax \int_0^x \exp( \mu^{-1} \tau(x,y) ) | g(y) | \ d y  \right) \\
&\leq \mu^{-1} \| g \|_{\infty} \left( 1  + \mu^{-1} \sgmax \int_0^x \exp( \mu^{-1} \tau(x,y) ) \ d y  \right) \\
&\leq \mu^{-1} \| g \|_{\infty} \left( 1  + \sgmax \sgmin^{-1} \right) \ ,
\end{align*}
where the integral is bounded as in the proof of  \iggb{Lemma \ref{lem:Smustab}}. The proof for $\mu < 0$ is similar.
The proof that $\cSmu: L_\infty \mapsto \C$, with the given bound, is a  consequence of the fact that, given any $g \in L_\infty$, the derivative of $\cSmu g$ is bounded (except when $\mu = 0$),  and then using \iggb{Lemma \ref{lem:Smustab}}.
\end{proof}

\begin{lemma} 
\label{lem:SmuMUderiv}
For $g \in L_\infty$ and any $\beta > 0$, we have 
$$
\sup_{x \in [0,1]} \ \int_{-1}^1 \ | \mu |^\beta \bigg \lvert \frac{\partial}{\partial \mu} (\cSmu g)(x) \bigg \rvert \ \rd \mu \ \leq \ 2 \beta^{-1} \sgmin^{-1} \| g \|_\infty \ .
$$
\end{lemma}

\begin{proof}
Consider $\mu > 0$. By the product rule
\begin{equation} \label{eq:prodrule1}
\frac{\partial}{\partial \mu} \left( \mu^{-1} \exp(\mu^{-1} \tau(x,y)) \right) = - \mu^{-2} \left( 1 + \mu^{-1} \tau(x,y) \right) \exp( \mu^{-1} \tau(x,y) )  \ .
\end{equation}
Using the definition of $\cSmu g$, \eqref{eq:prodrule1} and the
substitution $x \mapsto  t$ given by $t = \mu^{-1} \tau(x,y)$ for each $y$, we have
\begin{align*}
- \frac{\partial}{\partial \mu} (\cSmu g)(x) &= \mu^{-2} \int_0^x [ 1 + \mu^{-1} \tau(x,y) ] \exp[\mu^{-1} \tau(x,y)] g(y) \ d y \\
&= \mu^{-1} \int_{\mu^{-1} \tau(x,0)}^0 (1 + t) \ \exp(t)  \sigma^{-1}(\tau_x^{-1}(\mu t)) \ g(\tau_x^{-1}(\mu t)) \ d t \ ,
\end{align*}
where $\tau_x$ denotes the function $\tau_x(y) = \tau(x,y)$, with derivative $\tau_x'(y) = \sigma(y)$.  Thus,
\begin{align*}
\bigg \lvert \frac{\partial}{\partial \mu} (\cSmu g)(x) \bigg \rvert 
&\leq \mu^{-1} \| g \|_\infty \int_{\mu^{-1} \tau(x,0)}^0 ( 1 + |t| ) \exp(t)  \sigma^{-1}(\tau_x^{-1}(\mu t)) \ d t \ ,
\end{align*}
which is an integral with a positive integrand. Hence,
$$
\bigg \lvert \frac{\partial}{\partial \mu} (\cSmu g)(x) \bigg \rvert \ \leq \ \mu^{-1} \sgmin^{-1} \| g \|_\infty \int_{-\infty}^0 ( 1 + |t| ) \exp(t) \ d t \ = \ 2 \mu^{-1} \sgmin^{-1} \| g \|_\infty  \ . 
$$
This implies that, for any $\beta > 0$ and any $x \in [0,1]$,
$$
\int_{0}^1  | \mu |^\beta \bigg \lvert \frac{\partial}{\partial \mu} (\cSmu g) (x) \bigg \rvert \ \rd \mu \ \leq \ 2 \int_0^1 \mu^{-1 + \beta} \sgmin^{-1} \| g \|_\infty \ \rd \mu \ = \ 2 \beta^{-1} \sgmin^{-1} \| g \|_\infty \ .
$$
The estimate for the integral over  $\mu \in [-1,0)$ is analogous.
\end{proof}

The next few results will allow us to prove that
\[
{\sigma_S, \sigma, f \in \C^\eta_{pw}, \ \text{for any} \ \ \eta \in (0,1]
\quad \Rightarrow \quad \phi \in \C^\xi, \ \text{for all} \ \  
\xi \in (0,1).}
\]
This result uses the smoothing property of the operator $\cK$ obtained in Theorem \ref{thm:mapbnds} below. 

\begin{lemma}
\label{lem:IminKinvL2}
 On $L_2$, the operator $\cK \sigma_S$ is bounded and $\left( I - \cK \sigma_S \right)$ is invertible with the bound
\begin{equation}
\label{eq:IminK2inv}
\| \left( I - \cK \sigma_S \right)^{-1} \|_{\cLLtwo} \ \leq \ \left( \sgmax / \sgmin \right)^{1/2} \left( 1 - \| \sigma_S / \sigma \|_{\infty} \right)^{-1} \ .
\end{equation}
Moreover, the scalar flux $\phi \in L_2$ and
\begin{equation} \label{eq:phi2bndsrce}
\| \phi \|_2 \ \leq \ \left( \sgmax / \sgmin \right)^{1/2} \sgSmin^{-1}  \left(1 - \| \sigma_S / \sigma \|_\infty \right)^{-1} \| f \|_2 \ .
\end{equation}
\end{lemma}

\begin{proof}
 These results follow from  \cite{BlGrScSp:18}. In particular      \eqref{eq:IminK2inv} follows from
    \cite[Theorem 1]{BlGrScSp:18}, while  \eqref{eq:IminK2inv} is \cite[Corollary 3]{BlGrScSp:18}. The results are given
    there for the transport equation in 3D space and 2D in angle, but as explained in  \cite[Remark 2]{BlGrScSp:18}, the
    same results hold in the 1D space-angle problem considered here.  
\end{proof}

Before we continue with our results on the operator $\cK$, we must prove the following preliminary results relating to its integrand $E_1$.

\begin{lemma}
\label{lem:E1diff}
For any $x \in (0,1)$ let $\delta > 0$ be such that $x + \delta \in (0,1]$. For $y \in [0,1] \backslash [x,x+\delta]$,
$$
\bigg \lvert E_1(|\tau(x+\delta,y)|) \ - \ E_1(|\tau(x,y)|) \bigg \rvert \ \leq \ \delta \ \left( \frac{\sgmax}{\sgmin} \right) \ \frac{1}{\min \left\lbrace | y - (x + \delta) | , \ | y - x | \right\rbrace} \ .
$$
\end{lemma}

\begin{proof}
  Recall  that $-E_1'(z) =  \exp(-z)/z$ for $z \in \mathbb{R}^+$ \cite[eq.(5.1.26)]{AbSt:64}. Then,
  for $x \neq y$,
$$
\frac{d}{d x} E_1(|\tau(x,y)|) = -\frac{\exp(-|\tau(x,y)|)}{|\tau(x,y)|} \ \frac{d}{d x} |\tau(x,y)| = \sigma(x) \sgn(y-x) \frac{\exp(-|\tau(x,y)|)}{|\tau(x,y)|} \ .
$$
Thus, there exists $z \in (x,x+\delta)$ (which may depend on $y$) such that
$$
E_1(|\tau(x+\delta,y)|) - E_1(|\tau(x,y)|) \ = \ \delta \, \sigma(z) \, \sgn(y - z)  \frac{\exp(-|\tau(z,y)|)}{|\tau(z,y)|} \ .
$$
Now, if $y > x + \delta$, we have $ |\tau(c,y)| \ = \ \int_c^y \sigma(z) \ d z \ \geq \ \sgmin (y - x - \delta) > 0$, and hence,
$$
\bigg \lvert E_1(|\tau(x+\delta,y)|) \ - \ E_1(|\tau(x,y)|) \bigg \rvert \ \leq \ \delta \, \frac{\sgmax}{\sgmin} \frac{1}{(y - x - \delta)} \ .
$$
The result for $y > x + \delta$ follows, and the result for $y < x$ holds similarly.
\end{proof}

In the next proof, we shall use the expansion  (\cite[Eq.~(5.1.11)]{AbSt:64}):
\begin{equation}
\label{def:E1_expand}
E_1(z) = \log(z) - \Ein(z) - \gamma\,,
\end{equation}
where
\begin{equation} \label{eq:Ein}
\Ein(z) \ := \ \int_0^z \frac{1}{t} \left(1 - \exp(-t) \right) \ dt \ , \quad \text{ for all } z > 0 \,,
\end{equation}
and $\gamma$ is Euler's constant. Elementary calculus shows that $\Ein(z) \le z$, for all $z > 0$.
\begin{theorem} 
\label{thm:mapbnds}
The operator $\cK$ maps $L_2$ to  $L_\infty$ and $L_\infty$ to $\C^\xi$, for all $0 < \xi < 1$. Moreover, the following bounds hold:
\begin{enumerate}[(i)]
\item $\displaystyle \| \cK \|_{L_2 \mapsto L_\infty} \ \leq \ \sqrt{\log(2)} \sgmin^{-1/2} \ ;$  
\item $\displaystyle \| \cK \|_{L_\infty \mapsto \C^\xi} \ \leq \ \igg{c_\xi}\, \brov{\sgmax}/{\brun{\sgmin}}  \ ,$
\end{enumerate}
\igg{where $c_\xi$ may depend on $\xi$ and where $\brov{a}$ and $\brun{a}$ are defined in Notation \ref{not:bar}}.
\end{theorem}

\begin{proof}
\igg{Throughout the proof, $c$ denotes a generic constant which may depend on $\xi$}.
{\em (i)} \ {Let $g \in L_2$ and $x \in [0,1]$. Using \eqref{eq:Kdef} and the Cauchy-Schwarz inequality, we have} 
\begin{align}
  \label{oneint}
2 | \cK g(x) | \ = \ \bigg \lvert \int_0^1 E_1(|\tau(x,y)|) g(y) \ d y \bigg \rvert \ \leq \ \| g \|_2 \left( \int_0^1  E_1^2(|\tau(x,y)|) \ d y \right)^{1/2} \ .
\end{align}
Since $E_1^2$ is strictly positive and monotonically decreasing on $\mathbb{R}^+$, we have (recalling \eqref{eq:taubnd})
\begin{align}  
  \int_0^1 E_1^2(|\tau(x,y)|) \ d y \ & \leq \ \int_0^1 E_1^2(\sgmin |x-y|) \ \rd y \ \leq \  2 \int_0^{\max\{ x, 1-x\}} E_1^2(\sgmin r) \rd r \  \nonumber \\
                                       &
                                        < \ 2 \left(\int_0^{\infty} E_1^2(r) \rd r\right) \sgmin^{-1} = 4 \log(2) \sgmin^{-1} , 
  \label{eq:twoints} 
\end{align}
where we used \cite[(5.1.33)]{AbSt:64}. Combining this with
\eqref{oneint} gives estimate  (i).

\medskip

\noindent
{\em (ii)} \ Similarly to part (i), for any $g \in L_\infty$ and $x \in [0,1]$, we have 
\begin{align*}
  2 \lvert \cK g(x) \rvert & \leq \int_0^1 E_1(|\tau(x,y)|) \ d y\,  \| g \|_\infty  \
    \leq \left(\int_0^1E_1^2 (|\tau(x,y)|) \ d y\right)^{1/2} \,   \,
                             \| g \|_\infty \ ,
\end{align*}
where we used the Cauchy-Schwarz inequality. Then,  using the calculation in part (i), 
 \begin{align} \label{starstar} 
\| \cK g \|_\infty \leq \sqrt{\log(2)}\,  \sgmin^{-1/2} \| g
   \|_\infty. 
\end{align} 
To bound the H\"older-seminorm of $\cK g$, let $\xi \in (0,1)$ and, for any $0 \leq x < z \leq 1$, set $\delta = z-x$. Then, for any $\epsilon > 0$,  we can write 
\igg{\begin{align} 
  & \sup_{\substack{0 \leq x < z \leq 1} } \frac{| \cK g(z) - \cK g(x) |}{| z - x |^\xi}
  =  \bigg \lbrace
    \sup_{\substack{\begin{array}{c} 0 < \delta < 1-x,  \\ x \in [0,1] \end{array}}} \frac{| \cK g(x + \delta) - \cK g(x) |}{\delta^\xi} \ \bigg \rbrace\nonumber \\
  &  =     \sup_{\substack{\begin{array}{c} 0 < \delta < 1-x,  \\ x \in [0,1] \end{array}}} \max  \bigg \lbrace
  \sup_{\substack{0 < \delta < \epsilon }} \frac{| \cK g(x + \delta) - \cK g(x) |}{\delta^\xi} \ , \
  \sup_{\substack{\epsilon \leq \delta \leq 1-x} } \frac{| \cK g(x + \delta) - \cK g(x) |}{\delta^\xi}
  \bigg \rbrace
 \ .\label{eq:hnormbnd1}
\end{align}}
(The second term may be void.) In what  follows we set $\epsilon = (2 \sgmax)^{-1}$.

Consider  the second term on the right hand side of
\eqref{eq:hnormbnd1}. By the triangle inequality and \eqref{starstar} 
we obtain, uniformly in $x$, 
\begin{equation} \label{eq:hnormbndA}
   \frac{| \cK g(x + \delta) - \cK g(x) |}{\delta^{\xi}} \le 2 \epsilon^{-\xi} \| \cK g \|_{\infty}
    \le c  \sgmax^\xi   \sgmin^{-1/2}\| g \|_\infty \le c \, \frac{\overline{\sgmax}}{{\sgmin^{1/2} } } \, \| g \|_\infty  \ .
\end{equation}

To estimate the first term on the right hand side of \eqref{eq:hnormbnd1}
let $\delta \in (0,\epsilon)$ and let $B_{\delta}(x) = [x - \delta, x+ \delta]$. 
 Define $\mathcal{I}_{\delta}(x) := [0,1] \backslash B_{\delta}(x) $ and $\mathcal{J}_{\delta}(x) := B_{\delta}(x) \cap [0,1]$. Then, 
\begin{align}
2 \bigg \lvert \cK g(x + \delta) - \cK g(x) \bigg \rvert &= \bigg \lvert \int_0^1 F(x,y,\delta) g(y) \ \rd y \bigg \rvert \nonumber \\
&\leq \| g \|_{\infty}  \left( \int_{\mathcal{I}_{\delta}(x)} + \int_{\mathcal{J}_{\delta}(x)} \right) \lvert F(x,y,\delta) \rvert \ dy \ , \label{eq:hnormbndB}
\end{align}
where $F(x,y,\delta) := E_1 (|\tau(x+\delta,y) |) - E_1(|\tau(x,y)|)$. We write  the integral over $\mathcal{I}_{\delta}(x)$ in \eqref{eq:hnormbndB} as 
\begin{equation} \label{eq:hnormbndB2}
\int_{\mathcal{I}_{\delta}(x)} | F(x,y,\delta) | \ dy \ = \ \left( \int_0^{x-\delta} + \int_{x+\delta}^1 \right) | F(x,y,\delta) | \ dy \ .
\end{equation}
Again, one of these integrals could be void. 
When  $x \geq \delta$, we  use Lemma \ref{lem:E1diff}  to bound, for
all $\xi \in (0,1)$, the first integral on the right of \eqref{eq:hnormbndB2}  by
\begin{align*}
\int_0^{x-\delta} \lvert F(x,y,\delta) \rvert \ d y &\leq \left( \frac{\sgmax}{\sgmin} \right) \delta  \int_0^{x-\delta} (x-y)^{-1} \ d y \\
&= \left( \sgmax / \sgmin \right) \delta \left( \log\left( 1 / \delta \right) + \log x \right) \\
&\leq \left( \frac{\sgmax}{\sgmin} \right) \delta  \log\left(
  \frac{1}{\delta}\right) \ \leq \  c \frac{\sgmax}{\sgmin} \delta^\xi
  \ .
\end{align*}
The second integral on the right of
\eqref{eq:hnormbndB2} has the same estimate. Hence,
\begin{equation} \label{eq:hnormbndC}
\int_{\mathcal{I}_{\delta}(x)} | F(x,y,\delta) | \ \rd y \ \leq \ c  \, \frac{\sgmax}{\sgmin} \, \delta^\xi \ .
\end{equation}

For the integral over $\mathcal{J}_{\delta}(x)$ in \eqref{eq:hnormbndB}, \igg{we use the triangle inequality to write}  
\begin{align}
\int_{\mathcal{J}_{\delta}(x)} \lvert F(x,y,\delta) \rvert \ \rd y \ &\leq \ \int_{\mathcal{J}_{\delta}(x)} \lvert E_1(|\tau(x+\delta,y) |) \rvert \rd y + \int_{\mathcal{J}_{\delta}(x)} \lvert E_1(|\tau(x,y)|) \rvert \rd y .                                                                                                   \label{eq:hnormbndD} 
\end{align}

Using the expansion \eqref{def:E1_expand} for $E_1(z)$, the second integral on the right hand side of \eqref{eq:hnormbndD} can be estimated by
\begin{equation} \label{eq:split}
\int_{B_{\delta}(x)} \lvert E_1(|\tau(x,y)|) \rvert \ dy \leq
\int_{B_{\delta}(x)} \lvert \log(|\tau(x,y)|) \rvert \ dy +
\int_{B_{\delta}(x)} \lvert \Ein(|\tau(x,y)|) \rvert \ dy + 2 \gamma
\delta \ .
\end{equation}
Then, by \eqref{eq:taubnd} and using the fact that $\Ein(z) \in [0,z]$, for all $z > 0$, we can bound the second integral on the right hand side of \eqref{eq:split}
\begin{equation}
\label{eq:split2}
\int_{B_{\delta}(x)} \lvert \Ein(|\tau(x,y)|) \rvert \ dy \ \leq \ \int_{B_{\delta}(x)} |\tau(x,y)| \ \rd y \ \leq \ 2\sgmax \delta^2 \ {\le \ \delta},
\end{equation}
where in the last step we used $\delta < \epsilon = (2\sgmax)^{-1}$.
For the first integral on the right hand side of \eqref{eq:split}, we use again \eqref{eq:taubnd} and the fact that $|\tau(x,y)| \leq \sgmax |x-y| < \epsilon \sgmax \leq 1/2$ to bound
\begin{equation} \label{eq:lnwrite}
\lvert \log \lvert \tau(x,y) \rvert \rvert 
= {- \log \lvert \tau(x,y) \rvert \le - \log \big(\sgmin | y - x |\big)}\ .
\end{equation}
Thus, the first integral on  the right of \eqref{eq:split} can be bounded by integrating the right hand side of
  \eqref{eq:lnwrite} over $B_\delta(x)$, which gives
\begin{align}
  \int_{B_{\delta}(x)} \lvert \log \lvert \tau(x,y) \rvert \rvert \ d y \ & \le  \ -2 \int_0^\delta \log \big(\sgmin r \big) \ d r \ \le \ 2 \delta \left( 1 - \log\big(\sgmin\delta \big) \right)\nonumber \\
& 
= \ 2 \sgmin^{\xi -1}
  \delta^{\xi}((\sgmin \delta)^{1-\xi} - (\sgmin \delta)^{1-\xi} \log(\sgmin \delta)) \
  \leq \ c \sgmin ^{\xi-1} \delta^{\xi} \ .   \label{eq:split1}  
\end{align}
The last inequality follows since $\sgmin \delta \leq \sgmax \delta
\leq 1/2$ and since $r^{1-\xi} (1- \log(r)) $ is bounded on
$[0,1/2]$. Substituting  \eqref{eq:split2} and \eqref{eq:split1} into \eqref{eq:split},  we finally obtain an estimate for the second integral on the right-hand side of \eqref{eq:hnormbndD}: 
\begin{equation}
\label{eq:hnormbndE}
\int_{B_{\delta}(x)} \lvert E_1(|\tau(x,y)|) \rvert \ d y \ \le \ c \delta^{\xi} \left(\sgmin^{\xi-1} +1 \right) . 
\end{equation} 
The first integral on the right hand side of \eqref{eq:hnormbndD} can be bounded analogously. Thus, combining this with \eqref{eq:hnormbndB}, \eqref{eq:hnormbndC} and \eqref{eq:hnormbndD}, we have shown that,  for $\delta \leq \epsilon$, 
\begin{equation} \label{eq:hnormbndF}
 \frac{| \cK g (x + \delta) - \cK g (x)|}{\delta^\xi} \ \leq \ c \, \frac{\brov{\sgmax}}{\brun{\sgmin}} \, \| g \|_\infty \ ,
\end{equation}
uniformly in $x\in[0,1]$.

Note that \eqref{eq:hnormbndF} also implies that $\cK g \in C $ and \eqref{starstar} can also be interpreted as a bound on the uniform norm of $\cK g$. Hence, combining \eqref{starstar},  \eqref{eq:hnormbndA} and \eqref{eq:hnormbndF}, we have
$$
\| \cK g \|_\xi \ = \ \| \cK g \|_\infty + \sup_{x,z \in [0,1]} \frac{| \cK g(z) - \cK g(x) |}{| z - x |^\xi} \ \leq \ c \,
\frac{\brov{\sgmax}}{\brov{\sgmin}} \, \| g \|_\infty \ .
$$
\end{proof}

\begin{lemma}
\label{lem:IminKinv}
{The operator $\left( I - \cK \sigma_S \right)$ is invertible on $\C$} with the bound
\begin{equation} \label{eq:R3}
\| \left( I - \cK \sigma_S \right)^{-1} \|_{\cLc} \ \leq \ 2 \brov{\sgmax}^{1/2} \, \frac{\sgmax}{\sgmin} \, \left( 1 - \bigg \| \frac{\sigma_S}{\sigma} \bigg \|_{\infty} \right)^{-1} \ =: \ \cRone \ .
\end{equation}
Moreover, $\left( I - \cK \sigma_S \right)$ is {also invertible on $\C_{pw}$}, with the same bound as above.
\end{lemma}

\begin{proof}
Let $g \in \C$ and suppose that 
\begin{equation}
\label{eq:abc1}
\left( I - \cK \sigma_S \right) v = g, \quad \text{or equivalently that} \ \ v = \cK \sigma_S v + g.
\end{equation}
This allows us to apply a bootstrapping argument. By Lemma \ref{lem:IminKinvL2} (i), we have $v \in L_2$ and
\begin{equation}
\label{eq:abc1_nrm1}
\| v \|_2 \leq \Big( \frac{\sgmax}{\sgmin} \Big)^{1/2} \left( 1 - \left\| \frac{\sigma_S}{ \sigma} \right\|_{\infty} \right)^{-1} \| g \|_2 \leq \Big( \frac{\sgmax}{\sgmin} \Big)^{1/2} \left( 1 - \left\| \frac{\sigma_S}{ \sigma} \right\|_{\infty} \right)^{-1} \| g \|_\infty\ .
\end{equation}
Using \eqref{eq:abc1} again, this time together with Theorem \ref{thm:mapbnds}(i) we get $v \in L_\infty$. Finally, using \eqref{eq:abc1} with Theorem \ref{thm:mapbnds}(ii) we conclude that $v \in C$,  and that (using \eqref{starstar}),  
\begin{equation}
\label{eq:abc1_nrm2}
\| v \|_\infty  \leq  \| \cK \sigma_S v \|_\infty + \| g  \|_\infty  \leq \sgmin^{-1/2} (\sigma_S)_{\max} \| v \|_2  + \| g  \|_\infty \leq \sgmax^{1/2}  \Big( \frac{\sgmax}{ \sgmin} \Big)^{1/2} \| v \|_2  + \| g  \|_\infty \ .
\end{equation}
The bound in \eqref{eq:R3} follows on combining \eqref{eq:abc1_nrm1} and \eqref{eq:abc1_nrm2}.

Now suppose $g \in \C_{pw}$. Then $g \in L_2$ and the argument above holds verbatim to show that then $v \in \C_{pw}$ and that the bounds in \eqref{eq:abc1_nrm1} and \eqref{eq:abc1_nrm2} hold again.
\end{proof}

The final result of this {section follows} from \Cref{thm:mapbnds} and \Cref{lem:IminKinv}.
\begin{corollary}
\label{cor:phibnd}
Let $f \in \C_{pw}$, and let $\phi$ be the solution to {\eqref{eq:ie}. Then $\phi \in \C^\xi$, for any $\xi \in (0,1)$, and the following two} bounds hold:
\begin{equation}
\| \phi \|_\infty \ \leq \ c \sgmin^{-1/2} \cRone \| f \|_\infty \qquad \text{and} \qquad 
\| \phi \|_{\xi} \ \leq \ \igg{c_{\xi}} \cRtwo  \| f \|_\infty \ , \label{eq:phi_inf_xi} 
\end{equation}
where \igg{$c_\xi$ may depend on $\xi$,} $\cRone$ is defined in
\eqref{eq:R3} and 
\begin{equation} \label{eq:R2}
\cRtwo \ := \  {\brov{\sgmax}}^{1/2}  \, ({\brov{\sgmax}}/{\brun{\sgmin}})^{3/2} \, \cRone \ .
\end{equation}
\end{corollary}
\begin{proof}
  Recall  \eqref{eq:ie}, so we have $(I - \cK \sigma_S)\phi = \cK f$. From the  proof of Theorem \ref{thm:mapbnds}(ii),  $\cK f \in \C^\xi$ for any $\xi \in (0,1)$ and   $\Vert \cK f \Vert_\infty \leq c \sgmin^{-1/2} \Vert f \Vert_\infty$. Then Lemma \ref{lem:IminKinv} implies that
  $$ \Vert \phi \Vert_\infty \ \leq\  c\,  \sgmin^{-1/2} {\cRone}  \Vert f \Vert_\infty .$$
   To obtain the second bound, we use Theorem \ref{thm:mapbnds}(ii) again to obtain  
\[
\| \phi \|_{\xi}  \leq  \left\| \cK \left( \sigma_S \phi + f \right)  \right\|_{\xi}  \ \leq \ c  \, (\brov{\sgmax}/\brun{\sgmin}) \, \Big( (\sigma_S)_{\max}\| \phi \|_\infty  + \| f  \|_\infty \Big) 
\]
and then combine this with the first bound in \eqref{eq:phi_inf_xi}. \igg{(Again $c$ may depend on $\xi$.)}
\end{proof}

\section{Deterministic Error Estimate}
\label{sec:fulldiscrete}

We now return to estimating the error $\phi - \phi^{h,N}$ using \eqref{eq:IE3}. Introducing the operator
\begin{equation}
\label{eq:KNdef}
\cK^N g(x) \ := \ \frac{1}{2} \sum_{|k|=1}^N w_k \left(\cS_{\mu_k} g \right)(x) \ = \ \frac{1}{2} \int_0^1 E_1^N( \vert \tau(x,y) \vert ) g(y) \ d y \ ,   
\end{equation}
with $E_1^N(z) := \sum_{k=1}^N  w_k \mu_k^{-1}  \exp(- \mu_k^{-1} z)$ denoting the $N$-point quadrature approximation of the exponential integral \eqref{eq:exp_integral}, we can write \eqref{eq:IE3} as 
\begin{equation} \label{eq:phiminphih}
\phi - \phi^{h,N} \ = \ \left( I - \cK^{h,N} \cPh \sigma_S \right)^{-1} \left( e^N + e^{h,N} \right) \ ,
\end{equation}
where 
\begin{equation} \label{eq:eMN}
e^N := \left( \cK - \cK^{N} \right) \left( \sigma_S \phi + f \right) \ \
\text{and} \ \
e^{h,N} :=  \left[ \left( \cK^N - \cK^{h,N} \right) + \cK^{h,N} \left( I - \cPh \right) \right] \left( \sigma_S \phi + f \right)\,.
\end{equation}
Finally, to obtain an error estimate we apply the supremum norm to \eqref{eq:phiminphih}, and by trivial manipulation write
\begin{equation} \label{eq:error_norm}
\| \phi - \phi^{h,N} \|_{\infty} \ \leq \ \| \left( I - \cK^{h,N} \cPh \sigma_S \right)^{-1} \|_{\C \mapsto \C} \left( \| e^N \|_\infty + \| e^{h,N} \|_\infty \right) \ .
\end{equation}
The error analysis proceeds by showing that $\| e^N \|_\infty$ and $\| e^{h,N} \|_\infty$
both approach zero as $h \to 0$, $N \to \infty$ in an appropriate way and by finding a bound on $\| \left( I - \cK^{h,N} \cPh \sigma_S \right)^{-1} \|_{\C \mapsto \C}$.  The first is done in  Sections \ref{sec:angular_consist} and \ref{sec:spatial_consist}, while the second is done in Section \ref{sec:stability}.

\subsection{Consistency under Angular Discretisation}
\label{sec:angular_consist}

Here we estimate  $e^N$ using the angular regularity (\iggb{Lemma \ref{lem:SmuMUderiv}}) and the following result from De Vore and Scott \cite{DeSc:84} (see also \cite[Prop.~3.2]{PiSc:83}):

\begin{proposition}
\label{prop:quadbound}
Consider the $N$-point Gauss-Legendre rule on $[0,1]$ and let $m$ be a positive integer with $m \le 2N - 1$. Then we have
$$
\bigg \lvert \int_0^1 g(\mu) \ d \mu - \sum_{k=1}^N w_k g(\mu_k) \bigg \rvert \ \leq \ c N^{-m} \int_0^1 \left[\mu \left( 1 - \mu \right)\right]^{m/2} \lvert g^{(m)}(\mu) \rvert \ d \mu \ ,
$$
whenever the integral on the right hand side exists.
\end{proposition}

\begin{notation} \label{not:DGR}
In this paper the particular case of \eqref{eq:quad_generic} where the $N$-point Gauss-Legendre rule is used on both  $[-1,0]$ and on $[0,1]$, is called   the  {\em double Gauss rule}.
\end{notation}

\begin{theorem} 
\label{thm:eN}
Let $\cK^N$ be defined by \eqref{eq:KNdef} using the double Gauss rule. Then,
\begin{equation}
\label{eq:bound_angle}
\| \cK - \cK^N \|_{L_\infty \mapsto \C} \ \leq \ c \,  \sgmin^{-1} N^{-1} \ .
\end{equation}
\end{theorem}

\begin{proof}
Using \eqref{eq:Kdef}, \eqref{eq:KNdef}, the (anti-)symmetry properties of the double Gauss rule, then Proposition \ref{prop:quadbound} (with $m = 1$) and finally \iggb{Lemma \ref{lem:SmuMUderiv}} (with $\beta = 1/2$), we obtain for any $g \in L_\infty$, 
\begin{align*}
| \left( \cK - \cK^N \right) g(x) | \ &= \ \frac{1}{2} \bigg | \int_{0}^1 \left( \cSmu + \cS_{-\mu} \right) g(x) \ d \mu - \sum_{k=1}^N \left( w_k \cSmuk + w_{-k} \cS_{\mu_{-k}} \right) g(x) \bigg | \\
&\leq \ \frac{1}{2} \bigg | \int_{0}^1  \cSmu g(x) \ d \mu - \sum_{k=1}^N w_k  \cS_{\mu_k} g(x) \bigg | + \frac{1}{2} \bigg | \int_{0}^1 \cS_{-\mu} g(x) \ d \mu - \sum_{k=1}^N w_k \cS_{-\mu_{k}} g(x) \bigg | \\
&\leq \ c N^{-1} \left[ \int_0^1 \mu^{1/2} \bigg \lvert \frac{\partial}{\partial \mu} \left( \cSmu g(x) \right) \bigg \rvert \ d \mu + \int_{-1}^0 (-\mu)^{1/2} \bigg \lvert \frac{\partial}{\partial \mu} \left( \cSmu g(x) \right) \bigg \rvert \ d \mu \right]\\ 
&\leq \ c N^{-1} \int_{-1}^1 |\mu|^{1/2} \bigg \lvert \frac{\partial}{\partial \mu} \left( \cSmu g(x) \right) \bigg \rvert \ d \mu \ \leq \ \ c N^{-1} \sgmin^{-1} \| g \|_\infty \ .
\end{align*}
Hence, $\left( \cK - \cK^N \right): L_\infty \mapsto L_\infty$ satisfies the bound in \eqref{eq:bound_angle}. The extension to $\C$ holds because, by \iggb{Lemma \ref{lem:SmuXderiv}}, $\cSmu$ maps from $L_\infty$ to $\C$. 
\end{proof}

\begin{corollary}
\label{cor:eN}
Under the conditions of Theorem~\ref{thm:eN}, $e^N \in \C$,  with the bound
$$
\| e^N \|_\infty \ \leq \ c \frac{\brov{\sgmax}}{\sgmin} \sgmin^{-1/2}
\cRone N^{-1} \| f \|_\infty \ .
$$
\end{corollary}
\begin{proof}
By \eqref{eq:eMN}, Theorem~\ref{thm:eN}  and Corollary \ref{cor:phibnd}, we obtain
\begin{align*}
\| e^N \|_\infty \ \leq \ c  N^{-1} \sgmin^{-1} \left( \sgmax \| \phi \|_{\infty} + \| f \|_{\infty} \right)
\leq \ c  N^{-1} \sgmin^{-1} \left( \sgmax \sgmin^{-1/2} \cRone + 1 \right) \| f \|_{\infty} , 
\end{align*}
from which the estimate follows.
\end{proof}

\subsection{Consistency under Spatial Discretisation}
\label{sec:spatial_consist}
From now on, for convenience we make the following quasi-uniformity assumption on the spatial mesh:
\begin{assumption} \label{ass:QU}
  For some constant $\quasiunif \geq 1$, the local mesh diameters   $h_j : = x_j - x_{j-1}$ satisfy 
\begin{equation}
\label{eq:quasiunif}
\max_{j = 1, ...,  M} h_j \ =: \ h \ \leq \ \quasunif \min_{j = 1,...,M} h_j \ .
\end{equation}
\end{assumption}
\noindent We recall that when  $g \in \C_{pw}$, we have   
$\cPh g \in  \C_{pw}$ and   the following  elementary bounds:
\begin{equation}
\label{eq:perbnds} 
\| \cPh g \|_{\infty} \leq \| g \|_\infty\, \quad  \text{and} \quad \|
(I - \cPh) g \|_{\infty} \ \leq \ h^\xi \| g \|_{\xi,pw}\; ,  \
\text{ when} \ g \in \C^\xi_{pw} \; , \   \text{with} \   \xi\in(0,1].
\end{equation}

\begin{lemma}
\label{lem:Smuhstab}
Let $\mu \in [-1,1] \backslash \{ 0 \}$.
For $\cSmu^h$ defined by \eqref{eq:ptr_approx}, $\cSmu^h: L_\infty \mapsto V^h \subset \C$ and
\begin{align} \label{lemresult} 
\| \cSmu^h \|_{L_\infty \mapsto V^h} \ \leq \ 2 \quasunif \sgmin^{-1} \left( 1 + \sgmax \frac{h}{| \mu |} \right) \ .
\end{align} 
\end{lemma}

\begin{proof}
Without loss of generality, assume $\mu > 0$,  let $g \in L_\infty$ and let $u^h = \cS_\mu^hg$. Using the notation $\alpha_j = h_j \sigma_{j-1/2} / (2 \mu)$, by  \eqref{eq:ptr_approx} the nodal values of $u^h$  satisfy
\begin{equation} \label{eq:muhstab1}
\mu \left( 1 + \alpha_j \right) U_j = \mu \left( 1 - \alpha_j \right) U_{j-1} + \int_{I_j} g\ , \quad j=1,\ldots,M \ .
\end{equation}
Then, using the notation
$p_j = 1-\alpha_j$, $q_j = 1 + \alpha_j$ and $r_j = p_j / q_j$,
\eqref{eq:muhstab1} becomes
\begin{align} \label{eq:Uj}
U_j \ = \ r_j U_{j-1} + \frac{1}{\mu q_j} \int_{I_j} g \ .
\end{align}
To obtain a  bound on  $\vert U_j\vert $ for all $j$, suppose for the moment   
that  there exist $r, q \ge 0$ such that
\begin{equation} \label{eq:bndsrq}
\mathrm{(i)} \ \   | r_j | \leq \ r \ < \ 1\ ,  \quad \text{and} \quad  
\mathrm{(ii)} \ \ q_j \ \geq \ q \ \geq \ 1 \ ,  \quad \text{for all } \ 1 \leq j \leq M \ .
\end{equation}
Then a simple inductive argument on \eqref{eq:Uj} and using  $U_0 = 0$,
yields the estimates 
\begin{align}
  | U_j | \
           &
            \leq \
            \left(\sum_{i=1}^j r^{j-i}\right) \frac{h}{\mu q} \| g
             \|_{\infty} \,, \quad \text{for each} \ j,  \  \text{and} \quad \Vert \cS_\mu^h \Vert_\infty  \leq \ \frac{h}{\mu q (1 - r)} \| g \|_{\infty} \  . \label{eq:bnd123}
\end{align}

Finding  $q$ in \eqref{eq:bndsrq} is trivial, since  
$q_j \geq q := \max\{  1, \; \sgmin \frac{h}{2 \quasunif \mu}
\} \geq  1.$
Some  elementary but technical manipulations (details are in \cite[\S 3.3.2]{Pa:18})  show that \eqref{eq:bndsrq}(i) 
is satisfied,  with
\begin{equation} \label{eq:definer}
r \ = \ \begin{cases} 
\left(1 + \frac{\sgmin h}{2 \quasunif \mu} \right)^{-1} & \ \text{if } h \leq 2\mu / \sgmax \ , \\ 
\left(1 + \frac{2 \mu}{\sgmax h} \right)^{-1} & \ \text{if } h \geq 2 \quasunif \mu / \sgmin \  , \\ 
\left( 1 - \frac{\sgmin}{\sgmax \quasunif} \right) \ / \ \left(1 + \frac{\sgmin}{\sgmax \quasunif} \right) & \ \text{otherwise}  \ .
\end{cases}
\end{equation}

To finish the proof, we have to estimate $h/(\mu q(1-r))$ by the right-hand side of \eqref{lemresult}.
To illustrate this,  consider, for example, the case  $2 \mu / \sgmax \leq h \leq 2 \quasunif \mu  / \sgmin$. Then  $r$ is given by the third equation in \eqref{eq:definer} and  
$$
\frac{1}{q (1 - r)} \ \leq \ \frac{1}{2} \left( 1 + \quasunif \frac{\sgmax }{\sgmin} \right) \ = \
\frac{\quasunif }{\sgmin} \left( \frac{\sgmin}{2 \quasunif}  + \frac{\sgmax}{2} \right) \ < \ 
\frac{\quasunif }{\sgmin} \left( \frac{\mu}{h}  + \frac{\sgmax}{2}\right) \ .
$$
Multiplying each side by $h/\mu$ leads to the required estimate.   
The proof of \eqref{lemresult}   for the other values of $r$ is similar.
\end{proof}

\begin{lemma} 
\label{lem:uhuhat}
For $g \in L_\infty$, let $u = \cSmu g$,  $u^h = \cSmu^h g$, and let $\uhat$ denote
  the piecewise linear interpolant \igg{of}  $u$ on the mesh \eqref{eq:mesh}. Then 
$$
\| u^h - \uhat \|_\infty \ \leq \ \| \cSmu^h \|_{L_\infty \mapsto V^h} \left( \sgmax \| u - \uhat \|_\infty + h^\eta \| \sigma \|_{\eta,pw} \| u \|_{\infty} \right) \ . 
$$
\end{lemma}

\begin{proof} Recall that by (\Cref {ass:cross}), $\sigma \in \C^\eta_{pw}$, with  $\eta \in (0,1]$.
Using \eqref{eq:ptr_approx} and \eqref{eq:ptr}, 
\[
\int_{I_j} \mu \frac{d }{d x} u^h  + \wsigma \ u^h = \int_{I_j} g \ =\  \int_{I_j} \mu \frac{d}{d x}u  + \sigma u 
\]
and so
\begin{align}
\int_{I_j} \mu \frac{d }{d x} (u^h - \uhat)  + \wsigma \ ( u^h - \uhat) &= \int_{I_j} \mu \frac{d}{d x}(u - \uhat) + \sigma u - \wsigma \ \uhat\\ & =   \int_{I_j} \left( \wsigma \ (u - \uhat) + (\sigma - \wsigma) u \right) \ ,\label{eq:diff_FE_interp}
\end{align}
where the derivative {term in \eqref{eq:diff_FE_interp} vanishes because $u - \uhat$ vanishes at the mesh nodes}. Hence, 
\begin{align*}
\| u^h - \uhat \|_\infty  
&\leq \| \cSmu^h \|_{L_\infty \mapsto V^h} \left[ \| \wsigma \|_\infty \| u - \uhat \|_\infty + \| \sigma - \wsigma \|_\infty \| u \|_\infty \right] \ ,
\end{align*}
since by the definition of $\cSmu^h$ in \eqref{eq:ptr_approx} we have 
$u^h - \uhat \ = \ \cSmu^h \left[ \wsigma \ ( u - \uhat) + (\sigma -
  \wsigma) u \right]$. The result follows upon using  \eqref{eq:perbnds}.
\end{proof}

Our main deterministic error estimate (Theorem \ref{thm:phiMN}) will contain an $h \log N$ term.
The next result is the first indication of this:  the estimate in \eqref{eq:Smudiff} blows  up as $|\mu| \to 0$ for fixed $h$.

\begin{lemma} 
\label{lem:Smudiff}
There is a constant $c>0$, independent of all parameters, such that
\begin{equation} \label{eq:Smudiff}
\| \cSmu - \cSmu^h \|_{\LinftoC} \ \leq \ c \quasunif \sgmin^{-2} \left( \sgmax^2 \frac{h}{|\mu|} + \| \sigma \|_{\eta,pw} h^\eta \right) \ .
\end{equation}
\end{lemma}

\begin{proof}
For any $g \in L_\infty$, let $u = \cSmu g$, $u^h = \cSmu^h g$, and define  $\uhat$ as in Lemma~\ref{lem:uhuhat},
\begin{align}
  \nonumber \| \left( \cSmu - \cSmu^h \right) g \|_{\infty} \ &= \ \| u - u^h \|_{\infty} \ \leq \ \| u - \uhat \|_{\infty} + \| u^h - \uhat  \|_{\infty} \\
                                                              & \leq (1 + \sgmax \Vert \cS_\mu^h \Vert_{L_\infty \rightarrow V_h}) \Vert u - \uhat\Vert_\infty + h^\eta\Vert \cS_\mu^h \Vert_{L_\infty \rightarrow V_h} \Vert \sigma \Vert_{\eta, pw} \Vert u \Vert_\infty \;,                                                              
  \label{eq:abc2}
\end{align}
where we used Lemma \ref{lem:uhuhat}.  Now, from \iggb{Lemmas \ref{lem:Smustab} and  \ref{lem:SmuXderiv}}, 
we have
\begin{align}\label{eq:estuhat}  
\Vert u - \uhat \Vert_\infty \ \leq \ h \Vert u'\Vert_\infty \ \leq \
  2 \frac{h}{\vert \mu \vert } \frac{\sgmax}{\sgmin} \Vert g
  \Vert_\infty \quad \text{and} \quad \Vert u \Vert_\infty \ \leq \
  \sgmin^{-1} \Vert g \Vert_\infty \;. 
\end{align} 

First, consider the case $\sgmax h/\vert \mu \vert \leq 1$. Then Lemma \ref{lem:Smuhstab} gives
$\Vert \cS_\mu^h \Vert_{L_\infty \rightarrow V^h} \leq 4 \rho \sgmin^{-1} $. Combining  this and  \eqref{eq:estuhat} with   \eqref{eq:abc2} and applying  some elementary manipulation yields   \eqref{eq:Smudiff}.

If, on the other hand,   $\sgmax h / | \mu | > 1$, then by the triangle inequality and  \iggb{ Lemmas \ref{lem:Smustab}} and \ref{lem:Smuhstab} we have
\begin{equation} \label{eq:zzz}
\| \left( \cSmu - \cSmu^h \right) g \|_{\infty} \leq \| \cSmu g \|_{\infty} + \| \cSmu^h g \|_{\infty} \leq \sgmin^{-1} \| g \|_\infty + 2 \quasunif \sgmin^{-1} \left( 1 + \sgmax \frac{h}{|\mu|} \right) \| g \|_\infty \ ,
\end{equation}
which can also be bounded in the form \eqref{eq:Smudiff} by  using $\quasunif \geq 1$, $\frac{\sgmax}{\sgmin} \geq 1$ and $\frac{\sgmax h}{| \mu |} > 1$.
\end{proof}

The next lemma obtains some estimates needed  to bound $e^{h,N}$.

\begin{lemma} 
\label{lem:KminKMN}
Let $\cK^N$ and $\cK^{h,N}$ be defined by \eqref{eq:KNdef} and \eqref{eq:KMNdef} respectively, with $\mu_k$ and $w_k$ given by the double Gauss rule (Notation \ref{not:DGR}).
Under Assumption \eqref{ass:QU},  if $N \geq 2$, then
\begin{eqnarray*}
(i)  & \quad \| \cK^N - \cK^{h,N}\|_{L_\infty \mapsto \C} \ & \leq \ c \quasunif \sgmin^{-2} \left( \sgmax^2 h \log N + \| \sigma \|_{\eta,pw} h^\eta \right) \ ,\\
(ii)  & \quad \|  \cK^{h,N}\|_{L_\infty \mapsto \C} \ & \leq \ c \quasunif \sgmin^{-1}
\left( 1 +  \sgmax h \log N  \right) \ .
\end{eqnarray*}
\end{lemma}

\begin{proof}
Let  $g \in L_\infty$.  Using Lemma \ref{lem:Smudiff}, we have $ (\cK^N - \cK^{h,N} )g \in \C$ and 
\begin{align}
  \| \left( \cK^N - \cK^{h,N} \right) g \|_\infty &\leq c \sum_{|k|=1}^N w_k \| \cSmuk - \cSmuk^h \|_
                                                    {L_\infty \rightarrow \C}  \| g \|_\infty \nonumber \\
&\leq c \quasunif \sgmin^{-2} \sum_{|k|=1}^N w_k \left( \sgmax^2 \frac{h}{| \mu_k|} + \| \sigma \|_{\eta,pw} h^\eta \right) \, \| g \|_\infty  \ . \label{eq:endline}
\end{align}

Since the double Gauss rule integrates constants exactly, we have 
$\sum_{|k|=1}^N w_k = 2$. Also  \cite[Lemma 3.1]{PiSc:83} gives the estimate 
\begin{equation} \label{eq:sum_sing}
\sum_{|k|=1}^N w_k | \mu_k |^{-1} \ \leq \ c \left( 1 + | \log \mu_1 |
\right) \ \leq c \left( 1 + \log N \right) ,
\end{equation}
where the last inequality follows because $\mu_1 \sim N^{-2}$ for the Gauss  rule on $[0,1]$. 
Substituting these estimates into \eqref{eq:endline} yields the result (i). To obtain (ii), we proceed similarly:
\begin{align}
\| \cK^{h,N} g \|_\infty &\ \leq \ c \sum_{|k|=1}^N w_k \| \cSmuk^h \|_
                           {L_\infty \rightarrow \C} \,   \| g \|_\infty \ \leq \ c \rho \sgmin^{-1}
                          \sum_{|k|=1}^N w_k \left(1 + \sgmax \frac{h}{\vert \mu_k \vert}\right)  \| g \|_\infty \;, \nonumber 
\end{align}
from which the result follows analogously to (i). 
\end{proof}

\begin{theorem}
\label{thm:eMN}
Suppose the assumptions of Lemma \ref{lem:KminKMN} hold. Then, for $e^{h,N}$ defined in \eqref{eq:eMN}, 
$$
\| e^{h,N} \|_\infty \ \leq \ c \quasunif \sgmin^{-2} \brov{\| \sigma_S \|_{\eta,pw}}\cRtwo \left( \sgmax^2 h \log N + h^\eta  \| \sigma \|_{\eta,pw} \right) \| f \|_{\eta,pw} \ .
$$
\end{theorem}

\begin{proof}
  It is easy to check that
  $\Vert \sigma_S \phi\Vert_{\eta, pw} \, \leq \, 3 \Vert \sigma_S \Vert_{\eta, pw} \Vert \phi
  \Vert_{\eta} $. Using this,   Corollary \ref{cor:phibnd}, and recalling that $\cRtwo \geq 1$,
  we obtain
  \begin{align} \Vert \sigma_S \phi + f \Vert_{\eta,pw} \ \leq \
  \left( 1 + 3 \Vert \sigma_S \Vert_{\eta, pw} \cRtwo \right) \Vert f \Vert_{\eta,pw} \ \leq
    3 \overline{\Vert \sigma_S\Vert_{\eta,pw}} \cRtwo \Vert f \Vert_{\eta,pw} \;. \label{eq:th470}
  \end{align}
  Hence, using \eqref{eq:eMN} \eqref{eq:perbnds}, and then the results of Lemma \ref{lem:KminKMN}, 
  \begin{align}
    \Vert e^{h,N} \Vert_\infty \ & \leq \ \Vert \cK^N - \cK^{h,N} \Vert_{L_\infty \mapsto \C }
                                 \Vert \sigma_S \phi + f \Vert_\infty + \Vert \cK^{h,N} \Vert_{L_\infty \mapsto \C } \, h^\eta  \Vert \sigma_S \phi + f \Vert_{\eta,pw}\label{eq:th47}\\
    & \leq \ \left(\Vert \cK^N - \cK^{h,N} \Vert_{L_\infty \mapsto \C }
      + h^\eta  \Vert \cK^{h,N} \Vert_{L_\infty \mapsto \C } \right)
      \Vert \sigma_S \phi + f \Vert_{\eta,pw}\nonumber \\
                                 & \leq c \rho \sgmin^{-2}
                                   \left( \sgmax^2 h \log N +
                                   h^\eta \Vert \sigma \Vert_{\eta, pw}  + h^\eta
                                   \sgmin (1+ \sgmax h \log N )
                                   \right)
                                   \Vert \sigma_S \phi + f \Vert_{\eta,pw}\;. \nonumber 
  \end{align}
  The result is obtained by combining \eqref{eq:th470} and
  \eqref{eq:th47} and simplification. 
\end{proof}

\subsection{Stability and Convergence}
\label{sec:stability}

So far  we have shown \igg{that} $\| e^N \|_\infty$ and $\| e^{h,N} \|_\infty$
approach zero as $h \log N \to 0$ and $N \to \infty$. See
\Cref{cor:eN} and \Cref{thm:eMN}, respectively. To prove a final bound on \eqref{eq:error_norm} we need to show ``stability'', i.e. that  $\left( I - \cK^{h,N} \cPh \sigma_S \right)^{-1}$ exists and is
bounded in the $\| \cdot \|_{\C \mapsto \C}$ norm, independently of $h$ and $N$. To do this, a useful trick is to write
\begin{equation} \label{eq:clever}
\left( I - \cK^{h,N} \cPh \sigma_S \right)^{-1} \ = \ I \ + \ \cK^{h,N} \left(I - \cPh \sigma_S \cK^{h,N} \right)^{-1} \cPh \sigma_S \ ,
\end{equation}
which holds when the inverses on each  side exist. 
We obtain stability by proving that the inverse exists on the right-hand side and then
estimating all the terms on the right-hand side.

\begin{theorem}
  \label{lem:Nlarge}
Under the assumptions of Lemma \ref{lem:KminKMN}, there is a constant $K>0$ such that,  if $h$ and $N^{-1}$ are sufficiently small so that $h \log N \leq 1$
and  
\begin{equation} \label{eq:hNlarge}
  (h^{\eta} + h \log N + N^{-1})^{-1}  \ \geq\  K
  \left( \frac{(\sigma_S)_{\max}}{(\sigma_S)_{\min}}   \right)
\left( \frac{\brov{\sgmax}}{\brun{\sgmin}}\right)^{3} \brov{\| \sigma
  \|_{\eta,pw}} \cRone =: \cRthr \ ,
\end{equation}
then $\left( I - \cK^{h,N} \cPh \sigma_S \right)^{-1}$ is bounded on  $\C$,
with the bound
\begin{equation} \label{eq:pstab}
  \| \left( I - \cK^{h,N} \cPh \sigma_S \right)^{-1} \|_{\C \mapsto \C} \ \leq \ c \quasunif
  \left(\frac{\brov{\sgmax}}{\sgmin} \right) \, \left( \frac{\sgSmax}{\sgSmin} \right) \brov{(\sigma_S)_{\max}} \cRone \ =: \ \cRfr\ ,
\end{equation}
where $\cRone$ is defined in Lemma \ref{lem:IminKinv}.
\end{theorem}

\begin{proof}
 Introduce the family of  operators  $\mathcal{A}^{h,N} := I - \left( I - \sigma_S \cK \right)^{-1} \left( I - \cPh \sigma_S \cK^{h,N} \right)$.  Suppose that for some  $h$ and $N$, we can ensure   
\begin{equation} \label{eq:Abound}
\| \mathcal{A}^{h,N} \|_{\C_{pw} \mapsto \C_{pw}} \ \leq \ 1/2 \ .
\end{equation} 
Then it follows from the Banach Lemma that  
\begin{equation} \label{eq:hN1}
\| \left( I - \cPh \sigma_S \cK^{h,N} \right)^{-1} \left( I - \sigma_S \cK \right) \|_{\C_{pw} \mapsto \C_{pw}} \ = \ \| ( I - \mathcal{A}^{h,N} )^{-1} \|_{\C_{pw} \mapsto \C_{pw}} \ \leq \ 2 \ .
\end{equation} 
Therefore,
\begin{align}
\| \left( I - \cPh \sigma_S  \cK^{h,N} \right)^{-1} \|_{\C_{pw} \mapsto \C_{pw}} \ &= \ \| \left( I - \cPh \sigma_S  \cK^{h,N} \right)^{-1} \left( I - \sigma_S \cK \right) \left( I - \sigma_S \cK \right)^{-1} \|_{\C_{pw} \mapsto \C_{pw}} \nonumber \\
& \leq \ 2 \ \| \left( I - \sigma_S \cK \right)^{-1} \|_{\C_{pw} \mapsto \C_{pw}} \ \leq \ 2  \ \frac{\sgSmax}{\sgSmin} \cRone \ , \label{eq:partinv} 
\end{align}
where we used the identity $\left( I - \sigma_S \cK \right)^{-1} = \sigma_S \left( I - \cK \sigma_S \right)^{-1} \sigma_S^{-1}$ and Lemma \ref{lem:IminKinv}. Thus, on the assumption that
\eqref{eq:Abound} holds, we have (on combining  \eqref{eq:clever} with Lemma \ref{lem:KminKMN} (ii),
and \eqref{eq:partinv},     and  recalling $h \log N \leq 1$), 
\begin{equation} \label{eq:cleveruse}
  \Vert \left( I - \cK^{h,N} \cPh \sigma_S \right)^{-1}\Vert_{\C \mapsto \C}  \ \leq
  \ 1  \ + \ \left[c \rho \sgmin^{-1} ( 1 + \sgmax)\right]    \, \left[2  \ \frac{\sgSmax}{\sgSmin} \cRone\right] \, \left[{(\sigma_S)}_{\max}\right]  \ ,
\end{equation}
which yields \eqref{eq:pstab}. 

It remains to find conditions which ensure  \eqref{eq:Abound}. To do this, we  write
\begin{align}
\| \mathcal{A}^{h,N} \|_{\C_{pw} \mapsto \C_{pw}} \ &= \ \| \left( I - \sigma_S \cK \right)^{-1} \left[ \left( I - \sigma_S \cK \right) - \left( I - \cPh \sigma_S \cK^{h,N} \right) \right] \|_{\C_{pw} \mapsto \C_{pw}} \nonumber \\
&\leq \ \| \left( I - \sigma_S \cK \right)^{-1} \|_{\C_{pw} \mapsto \C_{pw}} \| \cPh \sigma_S \cK^{h,N} - \sigma_S \cK  \|_{\C_{pw} \mapsto \C_{pw}} \nonumber \\
&\leq \ \frac{\sgSmax}{\sgSmin} \| \left( I - \cK \sigma_S  \right)^{-1} \|_{\C_{pw} \mapsto \C_{pw}} \| \sigma_S \cK - \cPh \sigma_S \cK^{h,N} \|_{\C_{pw} \mapsto \C_{pw}} \nonumber \\
&\leq \ \frac{\sgSmax}{\sgSmin} \cRone \ \| \sigma_S \cK - \cPh \sigma_S \cK^{h,N} \|_{\C_{pw} \mapsto \C_{pw}} \ , \label{eq:hN2}
\end{align}
where we again used Lemma \ref{lem:IminKinv}. To estimate the right hand side of \eqref{eq:hN2} we write
\begin{align} \label{eq:hN3}
\| \sigma_S \cK  - \cPh \sigma_S \cK^{h,N} \|_{\C_{pw} \mapsto \C_{pw}} \ &\leq \ \| \left( I - \cPh \right) \sigma_S \cK \|_{\C_{pw} \mapsto \C_{pw}} \nonumber\\
&\ \ + \ \| \cPh \sigma_S \|_{\C \mapsto \C_{pw}} \left( \| \cK - \cK^{N} \|_{\C_{pw} \mapsto \C} + \| \cK^N - \cK^{h,N} \|_{\C_{pw} \mapsto \C} \right) \  \nonumber \\
& =: T_1 + T_2 \ .\nonumber\end{align}
We can bound $T_2$ using  \Cref{thm:eN} and \Cref{lem:KminKMN} to  obtain 
\begin{align*}  T_2 \ & \leq\ c \rho \, (\sigma_S)_{\max}
  \left( \sgmin^{-1} N^{-1} + \sgmin^{-2} \sgmax^{2} h \log N + \sgmin^{-2} \Vert \sigma\Vert_{\eta , pw}  h^\eta
                      \right) \ \\
                     & \leq \ c \rho \left(\frac{\brov{\sgmax}}{\brun{\sgmin}}\right)^3 \brov{\Vert \sigma\Vert_{\eta, pw}}(h^\eta + h
                       \log N + N^{-1}) . 
\end{align*}
(The last inequality is an over-estimate, but we do this to reduce technicalities.)
On the other hand, using {\eqref{eq:perbnds}} and \Cref{thm:mapbnds}, we have
$$ 
T_1 \ \leq \ c \igg{h^{\eta}} \Vert \sigma_S \Vert_{\eta, pw}
\frac{\brov{\sgmax}}{\brun{\sgmin}} \ .
$$
Combining the estimates for $T_1$ and $T_2$ we obtain
\begin{align*}
\Vert \mathcal{A}^{h,N} \Vert_{\C_{pw} \mapsto \C_{pw}} \ \leq \ c \rho  
  \left( \frac{(\sigma_S)_{\max}}{(\sigma_S)_{\min}}   \right)
\left( \frac{\brov{\sgmax}}{\brun{\sgmin}}\right)^{3} \brov{\| \sigma \|_{\eta,pw}} \cRone (h^{\eta} + h \log N + N^{-1})
\end{align*}
and the result follows on recalling \eqref{eq:Abound}. 
\end{proof}

We now have all the ingredients to prove the main result of this section.

\begin{theorem} \label{thm:phiMN}
Let $\phi^{h,N}$ be as defined in \S \ref{sec:quadrule}.  Under the assumptions of Lemma \ref{lem:KminKMN},
provided that $h \log N \leq 1$ and that \eqref{eq:hNlarge} holds, we have
\begin{equation}
  \| \phi - \phi^{h,N} \|_\infty \ \leq \ c \cR \left( N^{-1} \ + \ h \log N \ + \ h^{\eta} \right) \| f \|_{\eta,pw}
  \ ,
     \label{eq:detest} 
\end{equation} 
where
\begin{align} \label{eq:defR}
  \cR =  \quasunif \cRfr \cRtwo \left(\frac{\brov{\sgmax}}{\brun{\sgmin}}\right)^2
  \brov{\| \sigma \|_{\eta,pw}} \ \brov{\| \sigma_S \|_{\eta,pw}} \ .
\end{align} 
\end{theorem}

\begin{proof}
We employ  \eqref{eq:error_norm},  combined with  \Cref{lem:Nlarge},  \Cref{cor:eN} and \Cref{thm:eMN} to obtain:
\begin{align*}
  \| \phi - \phi^{h,N} \|_\infty \ &\leq \ c \quasunif \cRfr \left[ \frac{\brov{\sgmax}}{\sgmin}
                                     \sgmin^{-1/2} \cRone N^{-1} \right. \\
                                   & \quad\quad\quad\quad\quad\quad\quad  + \left. \frac{\brov{\| \sigma_S \|_{\eta,pw}}}{\sgmin^2} \cRtwo \left( \sgmax^2 h \log N + \| \sigma \|_{\eta,pw} h^\eta \right) \right] \| f \|_{\eta,pw}
\end{align*}
and the result follows after some algebra and recalling  $\cRone \leq \cRtwo$.
\end{proof}

\begin{remark} In \cite{Pa:18} a numerical example is given,  providing  evidence that the estimate in  
  \Cref{thm:phiMN} is sharp in terms of its dependence on the spatial smoothness parameter $\eta$.
\end{remark}

\section{Application in Uncertainty Quantification}
\label{sec:numerics}

In this section we allow  the cross-sections to be  random fields. Our main result is
Theorem \ref{thm:phiMNrand}, which is a probabilistic counterpart of  Theorem \ref{thm:phiMN}.
The chief technical difficulty in obtaining this is  the coefficient-dependent stability condition
\eqref{eq:hNlarge}, which in  the
random case becomes   a path-dependent  condition,    and so
simply integrating  \eqref{eq:detest} in probability space is not possible.    Instead we prove Theorem \ref{thm:phiMNrand}   for an {\em a priori}  chosen 
deterministic stepsize $h$. This means that for some realisations the mesh might need to be further refined in order to obtain stability. However in our cost estimate (Lemma \ref{lem:expect_cost}) we show that the expected value of the cost is unaffected by these (relatively rare) events.

\subsection{Random Input Data and Probabilistic Error Estimates}
\label{sec:random}

To  describe the random case, we let  $\omega \in \Omega$  denote a random event from a  sample space $\Omega$, and let   $\mathbb{P}: \Omega \mapsto [0,1]$  denote   the  associated probability measure.
For any normed space ($X$, $\| \cdot \|_X$),  we define the Bochner space $L_p(\Omega;X) := \lbrace g: \Omega \mapsto X \ : \   \| g \|_{L_p(\Omega;X)}^p :=  \int_\Omega \| g \|_{X}^p \rd \mathbb{P}(\omega) < \infty \rbrace$. 

\begin{assumption}{(Random Input Data)} \label{ass:random}
  We assume $\sigma_S = \sigma_S(\omega,\cdot)$, $\sigma = \sigma(\omega,\cdot)$
  and $f = f(\omega,\cdot)$ are now random fields.  We  set $\sigma_A(\omega, \cdot) = \sigma(\omega, \cdot) - \sigma_S(\omega, \cdot)$ and assume that $\sigma_S$, $\sigma_A$ and hence $\sigma$ are all positive-valued.
  Also, for an   {\em a priori} specified partition \eqref{eq:part},
  and for some $\eta \in (0,1)$ we assume:
\begin{enumerate}
\item[(a)] $\sigma, \ \sigma_S \in L_p(\Omega;\C_{pw}^\eta) \ $,   for all $p \in [1,\infty)$;
  \item[(b)] 
    $\ ((\sigma_S)_{\min})^{-1} \,  , \, 
  ((\sigma_A)_{\min})^{-1}  \in L_p(\Omega)$, for all $p \in [1,\infty)$;
\item[(c)] $f \in L_{p_{*}}(\Omega;\C_{pw}^\eta)$, for some $p_{*} \in (1, \infty]$.
\end{enumerate}
\end{assumption}
We note that (a) (b), combined with the positivitiy of the cross-sections imply that, for all $p \in [1,\infty)$, 
\begin{align} \label{eq:estmax}  (\sigma_S)_{\max} (\omega) \leq \sgmax(\omega)  \ & =  \ \Vert \sigma(\omega, \cdot)  \Vert_\infty \ \leq\
  \Vert \sigma(\omega, \cdot) \Vert_{\eta, pw} \in L_p(\Omega) \, ,   \\
  \text{and} \quad \quad  & \sgmin^{-1}  \leq ((\sigma_S)_{\min})^{-1}
                            \in L^p(\Omega) \, .
\end{align} 

\begin{example}\label{ex:matern} 
A class of suitable random fields that can be shown to satisfy
\Cref{ass:random}(a), (b) is got by choosing $\sigma_S(\omega, \cdot) =
\exp(Z(\omega, \cdot))$ where $Z$ is a centered Gaussian random field
with Mat{\'e}rn covariance function, and choosing $\sigma_A(\omega,
\cdot) = \sigma_A(\cdot)$ (a deterministic spatial function). Then,
\Cref{ass:random} (a), (b) hold true  with $\eta < \nu$ and $\nu $ is
the Mat{\'e}rn smoothness parameter \cite{ChSc:13}. This example will be used in the
numerical experiments in Section \ref{sec:prob_num}.
\end{example} 

Note that with these assumptions, the quantity   $\cR$  appearing on
the right-hand side of Theorem \ref{thm:phiMN}  is now a scalar-valued random variable.
Our next result,  Lemma \ref{lem:RinLp},  estimates the norm of this quantity in probability space.

\begin{lemma} \label{lem:RinLp}
$\cR \in \ L_p(\Omega)\,, \quad  \text{for all} \quad p \in \ [1,\infty).$
\end{lemma}

\begin{proof}
  Using equations \eqref{eq:defR}, \eqref{eq:pstab} and \eqref{eq:R2},
  we see that there exists an integer  $r>0$ such that the random variable
  $\cR$ has a pathwise bound of the form  
\begin{equation} \label{eq:boundform}
\cR \; \leq\;  {c} \, \left[ \cRdash \right]^r
    \,\left[\left( 1 - \left\| \frac{\sigma_S}{\sigma} \right\|_\infty \right)^{-1}\right]^2 
  \ \ \text{with} \ \
    \cRdash := 
  \frac{\brov{\sgmax}}{\brun{\sgmin}} \ 
  \brov{\| \sigma \|_{\eta,pw}}   \  \brov{\| \sigma_S \|_{\eta,pw}} \;.
\end{equation}

Each of the terms in $\cRdash$ can be shown to be in
$L_p(\Omega)$ for all $p \in [1, \infty)$. We justify this only
  for $\brov{\sgmax}/\brun{\sgmin}$. The other terms are similar.   
Recall from Notation \ref{not:bar}:  If  
    $a \in L_p(\Omega)$  is any scalar random variable, then  $\brov{a}\in L_p(\Omega)$.
    Also if $a^{-1} \in L_p(\Omega)$, then also $(\brun{a})^{-1} = \brov{(a^{-1})}\in L_p(\Omega)$. 

    Assumption \ref{ass:random} ensures that
    $\sigma_{\max} \leq \Vert \sigma \Vert _{\eta, pw}
    \in L_p(\Omega)$, and thus $\brov{\sigma_{\max}} \in L_p(\Omega)$,
    for all $p \in [1, \infty)$.
   Similarly, Assumption \ref{ass:random}(b) ensures
   $\brun{\sigma_{\min}}^{-1}  \in L_p(\Omega)$, for all $p \in [1,
   \infty)$. Using the generalised H\"{o}lder inequality, it follows
   that $\brov{\sgmax}/\brun{\sgmin} \in L_p(\Omega)$, for all
   $p \in [1,\infty)$.  Proceeding similarly for the other terms,  it
   follows that $\cRdash \in L_p(\Omega)$ for all $p \in
   [1,\infty)$.   

  To finish the proof we show that $\left( 1 - \| \sigma_S / \sigma
    \|_\infty \right)^{-1} \in L_p(\Omega)$, for all $p \in [0,\infty)$,  from which the result
  follows. First note that 
$$
0 \ < \ \frac{\sigma_S}{\sigma} \ = \ \frac{\sigma_S}{\sigma_S + \sigma_A} \ = \ 1 - \frac{\sigma_A}{\sigma_S + \sigma_A} \ \leq \ 1 - \frac{(\sigma_A)_{\min}}{\sgSmax + (\sigma_A)_{\max}} \ < \ 1 \ . 
$$
Then it follows that
\begin{align} \label{rhs}
\left( 1 - \bigg \| \frac{\sigma_S}{\sigma} \bigg \|_\infty \right)^{-1} \leq \ \left( 1 - \left( 1 - \frac{(\sigma_A)_{\min}}{\sgSmax + (\sigma_A)_{\max}} \right) \right)^{-1} = \  \frac{\sgSmax + (\sigma_A)_{\max}}{(\sigma_A)_{\min}} . 
\end{align} 
Now, by the Assumption \ref{ass:random}, $\sigma_A = \sigma- \sigma_S \in L_p(\Omega, \C_{\eta, pw})$ and so
$(\sigma_A)_{\max} \in L_p(\Omega)$, and so the numerator in
\eqref{rhs} is $L_p(\Omega)$ for all $p \in [1, \infty)$.   Assumption
\ref{ass:random}(b) allows us to estimate the denominator, and it
follows that  the second term in \eqref{eq:boundform} is also in
$L_p(\Omega)$ for all $p \in [1, \infty)$. The result follows
again via the generalised H\"{o}lder inequality.  
\end{proof}

We now establish a probabilistic counterpart of \Cref{thm:phiMN}.
To simplify the presentation, we assume the following
relationship of the angular and spatial discretization parameters.
\begin{assumption}\label{ass:Nh}
  For each mesh diameter  $h$,  we assume the number of angular quadrature points is $2N$,  with  
\begin{equation} \label{eq:Nh_relation}
N \ = \ N(h) \ = \  \lceil {c_0} h^{-\eta } \rceil\;, 
\end{equation}  
for some constant $c_0>0$ independent of $h$ and $\omega$, where $\eta \in (0,1)$ is given in \Cref{ass:random}. We assume also that $c_0$ is chosen large enough so that $\log N(h) \geq 1$.
\end{assumption}
As a result of this assumption it is easily
seen that $h \log N(h) \leq c h \log h^{-1}$, and hence
\begin{equation}
  \label{eq:NandHrel}
h \ \le  \ N(h)^{-1} + h \log N(h)  + h^{\eta} \ \le \ c' \, h^{\eta},
\end{equation}
for some (different) constant {$c'>1$} independent of $h$ and $\omega$.

Now recall the mesh-dependent stability condition \eqref{eq:hNlarge} and note that
\Cref{ass:random} does not ensure that 
$\cRthr \in L_\infty(\Omega)$.
Hence,  for any fixed mesh size $h>0$,  it is impossible to guarantee that   \eqref{eq:hNlarge}
is satisfied uniformly for  all samples $\omega \in \Omega$. To deal with this problem we have to consider
sample-dependent mesh sizes.

\begin{definition} \label{def:hbar} Let $h>0$ be a deterministically
  chosen mesh diameter. For each $\omega \in \Omega$, let $\hmax$ be
  the largest possible value of $h$
such that the stability condition \eqref{eq:hNlarge} is satisfied with $\sigma_S=\sigma_S(\omega,x)$, $\sigma = \sigma(\omega,x)$ and $N = N(\hmax)$ and let  
\begin{equation} \label{eq:stab_h_def}
h_\omega \ := \ \min \lbrace h, \hmax \rbrace . 
\end{equation}
\end{definition} 
Solving the problem with mesh diameter $h_\omega$ is guaranteed to be
stable and to have an accuracy determined by $h$. The resulting
  numerical solution (defined as in \S \ref{sec:quadrule}) is denoted: 
\begin{align} \label{eq:Phi}
\phistab(\omega,x) := \phi^{h_\omega,N(h_\omega)}(\omega,x)\; . 
\end{align} 

In  \Cref{thm:phiMNrand}, we quantify the accuracy of $\Phi^h$ as an approximation to $\phi$ in terms of the deterministic mesh diameter  $h$.
The question then arises: how significant is the set ``$\Omega_{\text{bad}}$'', containing the  samples $\omega$ which
have to be computed using a finer mesh (i.e. where  $\hmax < h$). We shall see
in \Cref{lem:expect_cost} that this set is small and decreases as $h \rightarrow 0$ in such a way as to
ensure that the expected cost is not affected by the path-dependent stability criterion.       

\begin{theorem}
\label{thm:phiMNrand}
Under Assumptions \ref{ass:random} and \ref{ass:Nh}, $\phistab(\omega,\cdot)$ exists for all $\omega \in \Omega$ and
for any $1\le p < r \leq p_{*}$, there exists a positive constant $C_{p,r} > 0$ such that
\begin{equation}
\label{eqn:phiMNrand}
\| \phi - \phistab \|_{L_p(\Omega;L_\infty)} \ \leq \ 
C_{p,r} \; \| f \|_{L_r(\Omega;\C_{pw}^\eta)} \ h^{\eta} \, .
\end{equation}
\end{theorem}

\begin{proof}
  Existence of $\Phi^h$ follows from the definition of $h_\omega$.   Then, using  Theorem \ref{thm:phiMN}
 together with \eqref{eq:NandHrel}  and H{\"o}lder's inequality, we obtain 
\begin{align*}
\| \phi - \phistab \|_{L_p(\Omega;L_\infty)}^p &= \ \cE\left[ \| \phi - \phi^{h_\omega,N(h_\omega)} \|_\infty^p \right] \ 
\ \leq \ {c^p} \cE\left[  \vert \cR(\omega)  \vert^p  \,   \| f(\omega,\cdot) \|^p _{\eta,pw} h_\omega^{p \eta} \right]   \\
&\leq \  {c}^p  \; \big(\cE\left[ \vert \cR (\omega)\vert ^{q} \right]\big)^{p/q} \; \Big(\cE \left[ \| f(\omega,\cdot) \|_{\eta,pw}^{r} \right]\Big)^{p/r}\,  \,h^{p \eta}  \;,
\end{align*}
where 
${q}^{-1} + {r}^{-1}  =  p^{-1}$,  i.e., $q = pr/(r-p)$.
The result follows with $C_{p,r} = c \Vert \cR \Vert_{L_{q}(\Omega)} $.  
\end{proof}

\begin{remark}[Uniform Random Input Data]
  \em
   If we   strengthen  \Cref{ass:random} by requiring parts (a), (b) to hold for $p = \infty$
  (i.e. uniformly bounded random fields), then  $\cRthr \in L_\infty(\Omega)$.  In this case there exists a deterministic
  $h_{\max} $ so that,   when  $h \leq h_{\max}$, stability condition
  \eqref{eq:hNlarge} holds uniformly over all samples, and we can choose $p=r=p_*$ in \eqref{eqn:phiMNrand}.  If, in addition,
  $p_{*} = \infty$ in \Cref{ass:random}(c), then $\| \phi - \phistab \|_{L_\infty(\Omega;L_\infty)} = \mathcal{O}(h^{\eta})$.
  However, such a strengthening would rule out important cases: \Cref{ass:random}(a), (b)   does not hold with $p=\infty$  for the lognormal fields  considered in \S \ref{sec:prob_num}. 
\end{remark}

For the general case of sample-dependent discretisations
it is important to discuss the \textit{expected computational cost} per sample.
The following lemma shows that if  the cost for computing each  sample is of order $h_\omega^{-\gamma}$ with a
sample dependent constant of proportionality which is sufficiently well-behaved and $h_\omega$ is given in
\eqref{eq:stab_h_def} then the expected cost per sample is $\mathcal{O}(h^{-\gamma})$. Essentially this shows that
the samples which need over-refinement to achieve stability
(i.e. $h_\omega^{\mathrm{max}} \ll h$ in \eqref{eq:stab_h_def}) are small
in measure.  In \Cref{ex:solvers} below, we give examples of solvers
for \eqref{eq:fulldisc}-\eqref{eq:phi_fulldisc} to which the  lemma can be  applied.  

 \begin{lemma} \label{lem:expect_cost}
Let Assumptions \ref{ass:random} and \ref{ass:Nh} hold and let $\Phi^h(\omega, \cdot)$  be  defined in \eqref{eq:Phi}.
   Assume   also that the cost
   $\mathcal{C}(\Phi^h)$ to compute one sample of  $\Phi^h$  is bounded by 
\begin{equation} \label{eq:rand_cost_bnd}
\mathcal{C}(\Phi^h) \ \leq \ C'(\omega) h_\omega^{-\gamma} \ ,
\end{equation}
with $C' \in L_p(\Omega)$,  for some $p>1$. Then \
$
\cE[ \mathcal{C}(\phistab) ] \leq   c  h^{-\gamma} .  
$
If $C' \in L_p(\Omega)$,  for some $p>2$, then we also have \
$
\V[ \mathcal{C}(\phistab) ] \leq   c  h^{-2\gamma} 
$.

\end{lemma}

\begin{proof} For any sample $\omega$,  $\hmax$ is defined as
  the largest stepsize which satisfies \eqref{eq:hNlarge}. Hence 
  inequality \eqref{eq:hNlarge} must fail if we replace $\hmax$ by $2 \hmax$, i.e.
  $$
  \Big((2 \hmax)^{\eta} + 2 \hmax  \log N(2 \hmax)  +
  N(2 \hmax)^{-1}
\Big)^{-1}  \ <  \cRthr(\omega) . 
$$
Using \eqref{eq:NandHrel}, this ensures that $\left( \hmax
\right)^{-{\eta}} < 2^\eta c' \cRthr$.
Then, using \eqref{eq:rand_cost_bnd} and \eqref{eq:stab_h_def},  we get
\begin{align}
\mathcal{C}(\phistab) \ \le \ C'(\omega) h_\omega^{-\gamma} 
&= \ C'(\omega) \max \lbrace h^{-\gamma} , \; \left( \hmax \right)^{-\gamma} \rbrace \nonumber\\
&\le \ C'(\omega)  \left(h^{-\gamma} +  2^\gamma \left( c' \mathcal{R}_3(\sigma(\omega,\cdot),\sigma_S(\omega,\cdot))\right)^{\gamma/\eta} \right)\ .
\label{eq:cost_est1}
\end{align}
Now, taking the expectation of \eqref{eq:cost_est1} and applying \igg{H\"older's} inequality
\begin{align*}
\cE \left[ \mathcal{C}(\phistab) \right] \ &\leq \ \cE\left[
                                             C'(\omega)  \right]
                                             h^{-\gamma} +  c \, \cE \left[ C'(\omega) \mathcal{R}_3(\sigma(\omega,\cdot),\sigma_S(\omega,\cdot))^{\gamma/\eta}  \right] \\
&\leq \ \cE\left[ C'(\omega)  \right] h^{-\gamma} +  c\Big(\cE \left[ C'(\omega)^p  \right]\Big)^{1/p} \Big(\cE \left[ \mathcal{R}_3(\sigma(\omega,\cdot),\sigma_S(\omega,\cdot))^{\gamma q/\eta}  \right]\Big)^{1/q} \ ,
\end{align*}
where $1 \leq q < \infty$ is such that $p^{-1} + q^{-1} = 1$ and
$c = 2^\gamma (c')^{\gamma/\eta}$. By a similar argument to that
in \Cref{lem:RinLp}, $\cRthr^{\gamma/\eta} \in L_q(\Omega)$ and so
the result follows, since $h\leq1$.

The variance result follows in the same way
  upon noting that $\V[ \mathcal{C}(\phistab) ] \le \cE \left[ \mathcal{C}(\phistab)^2 \right]$.
\end{proof}
\begin{remark} \label{rem:Omegabad}
  The second term on the right-hand side of \eqref{eq:cost_est1} is the contribution to the expected
  cost from samples for which the deterministically chosen mesh size $h$ is not stable, and for which further mesh  refinement is necessary (at least in theory). Since 
  $\hmax$ is chosen to be the {\em largest}  allowable mesh diameter which is stable,
  we can bound it both above and below in terms of     the $L_p$ integrable function $\cRthr$. Hence this second term remains bounded as $h \rightarrow 0$, giving the favourable complexity estimate proved in the lemma.   
  \end{remark} 

\begin{example}
\label{ex:solvers}
  In \cite{GrPaSc:17}, two methods for computing the solution to \eqref{eq:fulldisc}-\eqref{eq:phi_fulldisc} are presented. The corresponding system matrix is of dimension
  $\mathcal{O}(N(h_\omega) h_{\omega}^{-1} ) = \mathcal{O}(
  h_{\omega}^{-1 - \eta} )$.  

The first method is a direct solver
  where first $\psi$ is eliminated from the coupled system \eqref{eq:fulldisc}-\eqref{eq:phi_fulldisc} and
  then LU factorisation is applied to the resulting Schur complement system. The cost for this method is of order $h_\omega^{-2}\left(N(h_\omega)+h_\omega^{-1}\right)$ with a constant independent of $\omega$ \cite{GrPaSc:17}. Using \eqref{eq:Nh_relation} this implies that \eqref{eq:rand_cost_bnd} holds with $C'$ independent of $\omega$ and $\gamma = 3$.
  
The second method is a type of Richardson iteration known as
\textit{source iteration}. In that case, the cost is of order
$h_\omega^{-1}N(h_\omega)$ with a constant proportional to
  $-\big[\log \big( \|\sigma_S(\omega,\cdot)/\sigma(\omega,\cdot)\|_\infty \big)\big]^{-1}$ \cite{Bl:16,BlGrScSp:18,GrPaSc:17}. Using again \eqref{eq:Nh_relation}, this implies that \eqref{eq:rand_cost_bnd} holds with $\gamma= 1+\eta $.
\end{example}

\begin{corollary} \label{cor:solver_cost}
Suppose the assumptions of \Cref{lem:expect_cost} hold and system \eqref{eq:fulldisc}-\eqref{eq:phi_fulldisc} is solved with either of the Methods 1 or 2 in \Cref{ex:solvers}. Then, condition \eqref{eq:rand_cost_bnd} holds with $C' \in L_\infty(\Omega)$ (Method 1) and $C' \in L_p(\Omega)$, for all $1 \leq p < \infty$ (Method 2). Hence,
\begin{align*}
\cE[ \mathcal{C}(\phistab) ] \, &= \, \mathcal{O}\left(h^{-3} \right) \quad
                                  & \text{and} \quad \V[
                                    \mathcal{C}(\phistab) ] \, &= \,
                                    \mathcal{O}\left(h^{-6} \right)
                                    \quad &&\text{for Method 1, and}\\
\cE[ \mathcal{C}(\phistab) ] \, &= \, \mathcal{O}\left(h^{-(1+\eta)} \right) \quad
                                  & \text{and} \quad \V[
                                    \mathcal{C}(\phistab) ] \, &= \,
                                    \mathcal{O}\left(h^{-2(1+\eta)}
                                    \right) \quad &&\text{for Method
                                                    2,}
\end{align*}
respectively, and the hidden constants are independent of $h$. 
\end{corollary}
\begin{proof}
For the direct solver, $C'$ is independent of $\omega$, and hence trivially $C' \in L_{\infty}(\Omega)$.

In the case of the iterative solver,
$\|\sigma_S(\omega,\cdot)/\sigma(\omega,\cdot)\|_\infty \in (0,1)$, for almost all $\omega \in \Omega$, and since
$-\log(y)   > 1-y$, for all $y \in (0,1)$,  we have
\[
C'(\omega) \ \le \ - {c} \,\Big[ \log \left( \left\|\frac{\sigma_S(\omega,\cdot)}{\sigma(\omega,\cdot)}\right\|_\infty \right)\Big]^{-1} \
<  \ {c} \left( 1 - \left\|\frac{\sigma_S(\omega,\cdot)}{\sigma(\omega,\cdot)}\right\|_\infty \right)^{-1}. 
\]
As we have seen in \eqref{rhs},   this implies that $C' \in
L_p(\Omega)$, for all $p \in (0,\infty]$. The bounds on the expected
values and on the variances follow from \Cref{lem:expect_cost}. 
\end{proof}

For solving the linear systems arising from the transport equation in this paper we use standard source iteration, and this corollary  indicates the efficiency of this method for
the problems considered here. More efficient and flexible solvers such as those in \cite{KaRa:14} could be used in more general situations.

\subsection{Multilevel Monte Carlo Methods} \label{sec:MLMC_theory}

In this subsection we will \igg{consider} the application of MLMC techniques to compute functionals of the scalar flux $\phi$.
First we recall some general results on  MLMC. 
Suppose $Q = Q(\omega) $ is a random variable, whose expected value we wish to compute, and 
suppose $Q_h(\omega)$ is an approximation of
$Q(\omega)$ which becomes more accurate as the spatial mesh size $h \rightarrow 0$.
With   $\Qhat$ denoting an unbiased estimator for  $\cE[Q_h]$ (i.e. $\cE[\Qhat] = \cE[Q_h]$), 
the mean square error $e(\Qstabhat)$ in approximating $\cE[Q]$ with
$\Qhat$ is  given by
\begin{equation}
  e(\Qstabhat)^2 \
   := \ \cE\left[ (\Qstabhat - \cE[Q])^2 \right] \
= \ \left( \cE\left[ Q - \Qstab \right] \right)^2 \ + \ \V[\Qstabhat] \ ,
\label{eq:MSE}
\end{equation}
the first term being the square of the {\em bias} due to  discretization, and the second being  the {\em sampling error}
$\V[\Qstabhat] = \cE[ (\Qstabhat - \cE[\Qstabhat])^2 ]$.

In order to compare various estimators $\Qhat$, we define,  for any
$\epsilon \in (0,1)$,  the 
$\epsilon$-cost $\cC_{\epsilon}$ to be the number of (floating point) operations to achieve $e(\Qstabhat)^2 \le \epsilon^2$, a  sufficient condition for this being that  each of the terms on the right-hand side of 
\eqref{eq:MSE} should be bounded by  $\epsilon^2/2$. 

The standard Monte Carlo (MC) estimator for $\cE[Q]$ with $N_\MC$ samples is 
\begin{equation}
\Qstabhat^{MC} \ := \ \frac{1}{N_\MC} \sum_{n=1}^{N_\MC} \Qstab (\omega^{(n)}) \ .
\label{eq:estimatorMC}
\end{equation}
The sampling error is $\V[\Qstabhat^{MC}] = \V[\Qstab]/N_\MC$ and since $\V[\Qstab]$ is bounded as $h \to 0$, bounding this by  $\epsilon^2 / 2$ requires   $N_\MC \sim \epsilon^{-2}$.

To go further with the analysis one must make assumptions about the accuracy of the approximation $Q_h \approx Q$ and the
cost $\cC(\Qstab(\omega))$ of computing a single sample of $\Qstab$. Following \cite{Cl:11} we  
assume that there exist two constants $\alpha, \gamma > 0$ such that
\begin{align} \label{eq:bias_assmc}
\Big|\cE\left[ Q - \Qstab \right]\Big| \ & = \ \cO \left( h^\alpha \right) \ ,\\
\cE\left[\cC(\Qstab)\right] \ &= \ \cO \left( h^{-\gamma} \right) \ .\label{eq:cost_assmc}
\end{align}
Then,  to achieve an error of order $\epsilon$  in the ``bias'' \eqref{eq:bias_assmc},   we need to take $h
\sim \epsilon^{1/\alpha}$,  leading,  via \eqref{eq:cost_assmc}, to a mean  cost per sample of
order $\epsilon ^{-\gamma/\alpha}$. It follows immediately that there
exists a constant $c_\mu > 0$ independent of $\epsilon$ such that
the mean $\epsilon$-cost of the standard Monte Carlo estimator satisfies
\begin{equation} \label{eq:epscost_MC}
\cE\left[\cC_{\epsilon}(\Qstabhat^{MC})\right] \ = \ \cE\left[ \sum_{n=1}^{N_\MC}\cC(\Qstab(\omega^{(n)}))\right] \ = \ N_\MC \, \cE\left[\cC(\Qstab)\right] \ \le \ c_\mu \, \epsilon^{-2 - \frac{\gamma}{\alpha}}  \ .
\end{equation}
In fact,  this result can  be strengthened, as we now show.
\begin{theorem}\label{thm:probcost}
Let $\delta \in (0,1)$. In addition to \eqref{eq:bias_assmc} and
\eqref{eq:cost_assmc}, we assume that
\begin{equation} \label{eq:bnd_var_cost}
\V\left[\cC(\Qstab)\right] \ = \ \cO \left( h^{-2\gamma} \right)\, .
\end{equation}
Then there exist constants $c_\mu$ and $c_\sigma$ such that, 
for any $\epsilon \le \epsilon_{\max} \le 1$, 
$$
\mathbb{P} \left[ \cC_{\epsilon}(\Qstabhat^{MC}) < 
  (c_{\mu} + c_\sigma \epsilon_{\max} \delta^{-1}) \; \epsilon^{-2 -
  \frac{\gamma}{\alpha}} \right] \; > \; 1 - \delta^2.
$$
\end{theorem} 
\begin{proof}
Let  $X \in
  \mathbb{R}^+$ be any positive random variable with $\cE[X] = \mu $ and $\V[X] = \sigma^2 $. 
 For any $\delta \in (0,1)$, the
Chebyshev Inequality implies that
\begin{align}\label{eq:cheb}
1 - \delta^2 \; < \; \mathbb{P}\left[ |X-\mu| < \frac{\sigma}{\delta}
\right] \;
\le \; \mathbb{P}\left[ X < \mu + \frac{\sigma}{\delta} \right] \, . 
 \end{align}
Now, choosing $X = \cC_{\epsilon}(\Qstabhat^{MC})$  and using the choices
$N_{\mathrm{MC}} \sim \epsilon^{-2}$,
$h \sim \epsilon^{1/\alpha}$, an estimate for $\mu = \mathbb{E}[X]$ is given in \eqref{eq:epscost_MC}.   
Moreover, using in addition,  
\eqref{eq:bnd_var_cost},  we have 
\[
\sigma^2 \ = \ \V \left[ X \right] \ = \ \V\left[ \cC_{\epsilon}(\Qstabhat^{MC}) \right] = \V\left[
  \sum_{n=1}^{N_\MC}\cC(\Qstab(\omega^{(n)}))\right] \ = \ N_\MC \,
\V\left[\cC(\Qstab)\right] \ \le \ \left( c_\sigma \, \epsilon^{-1 -
    \frac{\gamma}{\alpha}} \right)^2 \ .
\]
The result then follows by inserting these estimates in \eqref{eq:cheb}.
\end{proof}

Thus, the $\epsilon$-cost of a particular realisation of the
standard Monte Carlo estimator $\Qstabhat^{MC}$ is $\cO
(\epsilon^{-2 - \frac{\gamma}{\alpha}}) $ with probability $1 -
  \delta^2$, for any $\delta \in (0,1)$, i.e.,
arbitrarily close to 1 and not just in mean. In general, the
  asymptotic constant $c_\delta := c_\mu + c_\sigma
  \epsilon_{\max}\delta^{-1}$ blows up, as $\delta \to 0$, but for 
$\epsilon_{\max} = \mathcal{O}(\delta)$, $c_\delta$ can be bounded independently of $\delta$.

To reduce the high cost of the  MC method, the  multilevel Monte Carlo (MLMC)
method uses a hierarchy of discrete models of increasing cost and
accuracy, corresponding to a sequence of decreasing discretisation
parameters $h_{0} > h_{1} > ... > h_{L}$. By choosing 
$h_{L} = h \sim \epsilon^{1/\alpha}$ as above, the most accurate model
on level $L$ is designed to provide full bias accuracy  of
$\cO(\epsilon)$. However, the samples on the coarser grids can be used
as control variates. Writing
$$
\cE[\Qstab] \ = \ \sum_{\ell=0}^L \cE[Y_\ell]\ , \qquad \text{where} \ \ Y_\ell := Q_{h_{\ell}} - Q_{h_{\ell-1}} \ \ \text{and} \ \ Q_{h_{-1}} := 0,
$$
each of the expected values on the right hand side is then estimated separately. In particular, using a standard MC estimator with $N_\ell$ samples for the $\ell$th term, we obtain the MLMC estimator
\begin{equation}
\Qstabhat^{MLMC} \ := \ \sum_{\ell = 0}^L \widehat{Y}^{MC}_\ell \ = \
\sum_{\ell=0}^L \frac{1}{N_\ell} \sum_{n=1}^{N_\ell} Y_\ell (\omega^{(\ell,n)}) \ .
\label{eq:estimatorMLMC}
\end{equation}
Here, the notation $\{\omega^{(\ell,n)}\}_{n=1}^{N_\ell} $ means that $N_\ell$  i.i.d. samples are chosen on level $\ell$,
independently from the samples on the other levels.
 
Since $Q_{h_{\ell}}$ and $Q_{h_{\ell-1}}$ were both assumed to
converge in mean
  to $Q$ as $h_{\ell-1} \to 0$, it follows that $\mathbb{E}[Y_\ell] \to
  0$. To achieve a reduced cost for the MLMC
  estimator we need the additional assumption that there exists a $\beta > 0$ such that 
\begin{equation}
\label{eq:var_assmc}
\V[Y_\ell] \ = \ \cO \left( h_{\ell}^{\beta} \right) . 
\end{equation}
For this reason, MLMC is often referred to as a \textit{variance reduction method}.
In Theorem \ref{prop:epsMC}  below, we shall give a simple sufficient condition for
  \eqref{eq:var_assmc} to hold in our
  context. 

The following theorem is
  a simple extension of \cite[Thm.~1]{Cl:11} to the random cost case (see also
  \cite{FKSS:14}). As in \cite{Cl:11}, we assume for simplicity that
  there exists a $q \in (0,1)$ such that 
\[
h_{\ell} = q h_{\ell-1}, \qquad \text{for all} \ \ \ell=1,\ldots,L. 
\]
\begin{theorem} \label{thm:MLMC}
  Assume that \eqref{eq:bias_assmc}, \eqref{eq:cost_assmc}, \eqref{eq:var_assmc} hold with $\alpha, \beta, \gamma > 0$
  and $\alpha \ge \frac12 \min\{\beta,\gamma\}$.
Then, for any $\epsilon < e^{-1}$, there exist choices $L
  \sim \log(\epsilon^{-1})$ and $\{ N_\ell \}_{\ell=0}^L$ such that $e(\Qstabhat^{MLMC})^2 \le \epsilon^2$ and
\begin{equation}
\cE \left[\cC_{\epsilon}(\Qstabhat^{MLMC})\right] \ \le \ c_\mu \,
  \epsilon^{- 2 -  \max \lbrace 0 , \ (\gamma - \beta) /\alpha \rbrace
  } \,, \ \ \text{for} \ \beta \not= \gamma\;,
\label{eq:epscost_MLMC}
\end{equation}
with $c_\mu > 0$ independent of $\epsilon$. For $\beta = \gamma$, we can achieve $\cE
\big[\cC_{\epsilon}(\Qstabhat^{MLMC})\big] \le c_\mu \, \epsilon^{-
    2}(\log \epsilon)^{2}$. 

\medskip

\noindent 
Let $\delta \in (0,1)$ and let us assume in addition
  that \eqref{eq:bnd_var_cost} holds. Then there exists a constant
  $c_\sigma>0$ independent of $\epsilon$ and $\delta$ such that
\begin{equation}
\mathbb{P} \left[ \cC_{\epsilon}(\Qstabhat^{MLMC}) \le
(c_{\mu} + c_\sigma \delta^{-1}) \, \epsilon^{- 2 -  \max \lbrace 0 , \ (\gamma - \beta) /\alpha \rbrace } \right] \; > \; 1 - \delta^2\,, \ \ \text{for} \ \beta \not= \gamma\;, \label{eq:epscost_MLMC_prob}
\end{equation}
with $c_\mu > 0$ as above. For $\beta = \gamma$, we have $\mathbb{P} \left[ \cC_{\epsilon}(\Qstabhat^{MLMC}) \le
(c_{\mu} + c_\sigma \delta^{-1}) \, \epsilon^{- 2 }(\log \epsilon)^{2} \right] >  1 - \delta^2$.
\end{theorem}
\begin{proof}
As in \cite{FKSS:14}, the proof of \eqref{eq:epscost_MLMC} follows
easily from \cite[Append.~A]{Cl:11}. Due to the
independence of  the samples $\omega^{(\ell,n)}$, there exists
$c_\mu' > 0$ independent of $\epsilon$ such that
\begin{align} 
\cE\left[\cC_{\epsilon}(\Qstabhat^{MLMC})\right] \ 
&= \ \cE\left[
  \sum_{\ell=0}^L
  \sum_{n=1}^{N_\ell}\cC\left(Q_{h_{\ell}}(\omega^{(\ell,n)})\right) +
  \cC\left(Q_{h_{\ell-1}} (\omega^{(\ell,n)})\right) \right]
\ \le \ c_\mu' \sum_{\ell=0}^L N_\ell h_\ell^{-\gamma}\;,
\label{eq:epscost_bound_MLMC}
\end{align}
i.e., the same asymptotic bound as in the deterministic case, and the
result follows as in \cite{Cl:11} with identical choices for $L$ and $\{ N_\ell \}_{\ell=0}^L$.

To prove \eqref{eq:epscost_MLMC_prob}, we exploit again the
  independence of the samples and show as in
  \eqref{eq:epscost_bound_MLMC}, with the same values for $L$ and 
 $\{ N_\ell \}_{\ell=0}^L$, that there exist
$c_\sigma' , c_\sigma > 0$ independent of $\epsilon$ such that
\begin{equation}
\label{eq:epscost_varbound_MLMC}
\V \left[\cC_{\epsilon}(\Qstabhat^{MLMC})\right] \ \le \ (c_\sigma')^2
\sum_{\ell=0}^L N_\ell h_\ell^{-2\gamma} \ \le (c_\sigma)^2
\left\{ \begin{array}{ll} 
\epsilon^{- 2 -  \max \lbrace 0 , \ (2\gamma - \beta) /\alpha \rbrace}
          \ & \text{for} \ \beta \not= 2 \gamma\,,\\[1ex]
\epsilon^{- 2 }(\log \epsilon)^{2} \ & \text{for} \ \beta = 2 \gamma\,.
\end{array} \right.
\end{equation}
The second estimate in   \eqref{eq:epscost_varbound_MLMC} follows as
in \cite[Theorem 1]{Cl:11} after  replacing $\gamma$
  with $2\gamma$.

The  result \eqref{eq:epscost_MLMC_prob}  then follows as in the proof of Theorem
\ref{thm:probcost} via Chebyshev's Inequality. 
To see this, consider \eqref{eq:cheb} with $X =
\cC_{\epsilon}(\Qstabhat^{MLMC})$ and assume first that $\beta <
\gamma$. Using the bounds on the expected value
and variance of $X$ in \eqref{eq:epscost_MLMC}  and
\eqref{eq:epscost_varbound_MLMC} we get 
\[
\mathbb{P} \left[ \cC_{\epsilon}(\Qstabhat^{MLMC}) \le
 c_\delta \, \epsilon^{- 2 -  (\gamma - \beta) /\alpha} \right] \; > \; 1
- \delta^2
\]
with
$$
c_\delta \ := \ c_\mu + 
c_\sigma \epsilon^{1-\frac{\beta}{2\alpha}} \, \delta^{-1} \ \le \
c_\mu + c_\sigma \delta^{-1} \,,
$$
since $\epsilon < e^{-1}$ and $\alpha \ge \frac12 \beta$.

The cases
$\beta=\gamma$, $\beta \in (\gamma,2\gamma)$, $\beta = 2\gamma$ and
$\beta > 2\gamma$ can all be shown similarly.
\end{proof}

Theorem \ref{thm:MLMC} states that provided \eqref{eq:var_assmc} holds
for some $\beta > 0$, the MLMC always achieves a gain of
$\cO\left(\epsilon^{-\min\{\beta,\gamma\}/\alpha}\right)$ in the
asymptotic cost over
standard Monte Carlo, even in the case when the cost per sample is
random and with
probability arbitrarily close to 1. For sufficiently large $\beta$,
the cost of the MLMC method is $\cO\left(\epsilon^{-2}\right)$,
independent of $\alpha$. This fact can be exploited to design unbiased
multilevel estimators of  $\cE [ Q]$ with cost
$\cO\left(\epsilon^{-2}\right)$ \cite{RhGl:15}. On the other hand, if
$\gamma > \beta = 2 \alpha$, the cost of the MLMC method is
$\cO\left(\epsilon^{-\gamma/\alpha}\right)$ which is optimal, in the
sense that it is equivalent (up to a constant) to the cost of
computing a \textit{single (standard) Monte Carlo sample} to  accuracy
$\cO(\epsilon)$.

\medskip

\noindent {\bf Application to Neutron Transport.} \ \
Suppose now that $\Q: \C \rightarrow \mathbb{R}$ is a (linear or nonlinear)  functional (operating with respect to
the spatial variable $x$),  and we are interested in computing the expected value of $ Q(\omega) : = \Q(\phi(\omega, \cdot))$ where $\phi $ is the scalar flux satisfying \eqref{eq:transport_det} and \eqref{eq:bc}.  
This will be approximated by   $Q_h(\omega) : =  \Q(\Phi^h(\omega, \cdot))$, with $\Phi^h$ as defined in Theorem
\ref{thm:phiMNrand}.

Given the clear importance of the parameters $\alpha$, $\beta$, $\gamma$, we would now like to estimate them
theoretically. We have already estimated the parameter $\gamma$ for two different solvers in \Cref{cor:solver_cost}, 
taking into account the sample-dependent mesh size.
The following result gives estimates for   $\alpha$ and $\beta$ under a quite general assumption on $\Q$.

\begin{theorem} \label{prop:epsMC} Make the same  assumptions as in  Theorem \ref{thm:phiMNrand} but assume 
  $ p_{\ast}>2$ in \Cref{ass:random}. Let $2<p<p_{\ast}$    and let  $q = 2p/(p-2)$.  
  Suppose, in addition, $\Q$ satisfies the Lipschitz condition:     $$ \vert \Q(\phi(\omega, \cdot)) - \Q(\tilde{\phi}(\omega, \cdot)) \vert \ \leq C'(\omega) \Vert \phi(\omega, \cdot) - \tilde{\phi}(\omega, \cdot) \Vert_\infty ,\quad \text{for all} \quad \phi, \ 
   \tilde{\phi} \in L^p(\Omega, L_\infty)\;,$$
   where $C' \in L^q(\Omega)$.  
   Then,  \eqref{eq:bias_assmc} and \eqref{eq:var_assmc} hold with
$$
\alpha \; = \;  \eta,  \quad \text{and} \quad \beta \; = \; 2 \eta \  .
$$
\end{theorem}

\begin{proof}  
From the given hypothesis and H\"{o}lder's inequality 
  $$ \cE[\vert Q(\omega)  - Q_h(\omega) \vert ] \  =  \
  \cE[\vert \Q(\phi(\omega, \cdot))  - \Q(\Phi^h(\omega,\cdot)) \vert ]
  \ \leq \ \Vert C'\Vert_{L^{q'}(\Omega)} \,  \Vert \phi  - \Phi^h \Vert_{L^p(\Omega, L_\infty)}, $$
  where $q' = p/(p-1) < q$ and  \eqref{eq:bias_assmc} with $\alpha = \eta$ follows from Theorem \ref{thm:phiMNrand}. 
 
 Also, with $Y_\ell = Q_{h_\ell} - Q_{h_{\ell-1}}$, we have 
\begin{align*}
\V[ Y_\ell] \ \leq \ \cE\big[ Y_\ell^2\big] \ \leq \ 2 \left( \cE\left[ \left| Q - Q_{h_{\ell}} \right|^2\right] + \cE\big[ \left| Q - Q_{h_{\ell-1}} \right|^2 \big]\right).
\end{align*}
Arguing as before,
$$\cE\left[ \left| Q(\omega) - Q_{h_{\ell}}(\omega)  \right|^2\right]  = \cE\left[ \left| \Q(\phi(\omega, \cdot)  - \Q(\Phi^{h_\ell}(\omega, \cdot)) \right|^2\right] \leq \ \Vert C'\Vert^2_{L^{q}(\Omega)} \,
\Vert \phi  - \Phi^{h_\ell} \Vert_{L^p(\Omega, L_\infty)}^2, $$
where we used H\"{o}lder's inequality with conjugate indices $p/2$ and $q/2$. Then
\eqref{eq:var_assmc} follows with $\beta = 2 \eta$. 
\end{proof} 

\begin{example} \label{ex:nlf} 
Consider the $q$th moment of the spatial average of $\phi$ (over $[0,1]$):
\begin{equation} \label{eq:qoi}
Q(\omega) \ = \ \| \phi \|_1^q \ := \ \left( \int_0^1 |\phi(\omega,x)| \ dx \right)^q \ , \quad \text{ for some integer }  q \geq 1 \ .
\end{equation}
This satisfies the assumptions of \Cref{prop:epsMC}. The details are given in \cite{Pa:18}. 
\end{example} 

\begin{corollary} \label{cor:epscost}
  Suppose the assumptions of \Cref{prop:epsMC} hold and system \eqref{eq:fulldisc}-\eqref{eq:phi_fulldisc} is solved with Method 2 in \Cref{ex:solvers}.  Then, the $\epsilon$-costs of the Monte Carlo method and of the  multilevel Monte Carlo
  method satisfy, respectively,
$$
\cE\Big[\cC_{\epsilon}(\widehat{Q}_{h}^{MC})\Big] \ = \ \cO \big( \epsilon^{-4-\chi} \big)  \ \ \ \text{ and } \ \ \ \cE\Big[\cC_{\epsilon}(\widehat{Q}_{h}^{MLMC})\Big] \ = \ \cO \big( \epsilon^{- 2 -\chi} \big)\;, 
$$
where $\displaystyle \chi := \frac{1-\eta}{\eta} > 0$.
\end{corollary}

This corollary shows that indeed in theory, for the neutron transport
problem, a theoretical gain in the asymptotic computational cost of up
to two orders of magnitude in $\epsilon^{-1}$ is possible
on average. In fact, since we have also established a bound on the variance of the
cost of Methods~1 and 2 in Corollary \ref{cor:solver_cost}, we could
even deduce such a result with probability arbitrarily close to 1 from
Theorems \ref{thm:probcost} and \ref{thm:MLMC}.
However, in the numerical section we will see that this theoretical
  result is overly optimistic, since in particular the bound on
  $\alpha$ in Theorem \ref{prop:epsMC} is not sharp. Nevertheless, we
  do observe gains of (at least) one order of magnitude.
  
Similar results can also be shown for other functionals of $\phi$ that
are bounded in $L^p(\Omega,L_\infty)$.

\subsection{Numerical Results} 
\label{sec:prob_num}

In this section, we give some numerical experiments for the case when
$f(x) = e$, for all $x \in (0,1)$, and when $\sigma = \sigma_S +
\sigma_A$ with  fixed absorption cross-section   $\sigma_A =
\exp(0.5)$ and scattering cross-section $\sigma_S$  chosen to be
a lognormal random field with Mat\'{e}rn covariance, i.e.,
$\log \sigma_S$ is a centred  Gaussian random field, with covariance function: 
\begin{equation}
\label{eq:covfunc}
C_\nu (x,y) \ = \ \sigma_{var}^2 \frac{2^{1-\nu}}{\Gamma(\nu)} \left( 2 \sqrt{\nu} \frac{|x - y|}{\lambda_C}  \right)^\nu K_\nu \left( 2 \sqrt{\nu} \frac{|x - y|}{\lambda_C} \right) \ . 
\end{equation}
Here, $\Gamma$ is the gamma function,  $K_\nu$ is the modified Bessel function of the second kind,
$\lambda_C$ is the correlation length
 and $\sigma_{var}^2$ is the variance.  
By increasing the positive  parameter $\nu$ we can  increase  the  smoothness of realizations (see, e.g. \cite{GrKu:15}). 

To sample from $\sigma_S$ we use the Karhunen-Lo{\`e}ve (KL) expansion of $\log \sigma_S$\,, i.e.,
\begin{equation}
\label{eq:KLexp}
\log \sigma_S (x,\omega) \ = \ \sum_{i=1}^{\infty} \sqrt{\xi_i} \ \eta_i(x) \ Z_i(\omega) \ ,
\end{equation}
where $Z_i \sim \mathcal{N}(0,1)$ i.i.d., and  $\xi_i$ and $\eta_i$ are the eigenvalues and the $L_2(0,1)$-orthogonal eigenfunctions of the integral operator induced by the kernel \eqref{eq:covfunc}. In practice, the KL expansion needs to be truncated after a finite number of terms,  and the accuracy of the  truncated expansion  depends on the
rate of decay of the eigenvalues -- this rate  gets faster as $\nu$ increases -- see, e.g. \cite{Lord:14, GrKu:15}.

We will give experiments for the cases $\nu = 0.5$ (when the $\xi_i$ and $\eta_i$ are known  analytically \cite{Lord:14}),  and $\nu = 1.5$ (where $\xi_i$ and $\eta_i$ are computed using the Nystr{\"o}m method - see, for example,  
\cite{EiErUl:07}).
When $\nu = 0.5$, it is known that Assumption \ref{ass:random} holds with $\eta < 0.5$. When  $\nu = 1.5$
then realizations of $\sigma_S$ have H\"{o}lder continuous first derivative,    and
hence Assumption \ref{ass:random} holds for all $\eta < 1$ -- see, e.g.,   \cite{ChSc:13,GrKu:15}. In our experiments we set  $\lambda_C = 1$ and $\sigma_{var}^2 = 1$.

We discretise using the method described in \S \ref{sec:quadrule}. Following \eqref{eq:Nh_relation}, we  set  the angular discretisation level at $N = 2\lceil 2 h^{-1/2} \rceil$ when $\nu = 0.5$ and
$N = (2h)^{-1}$ when $\nu = 1.5$. We truncate the  KL expansion after
$225 \lceil h^{-1/2}\rceil $ terms when $\nu = 0.5$ and after  $8
\lceil h^{-1}\rceil $ terms when $\nu = 1.5$, which ensures in both
cases that the truncation error is negligible compared to the discretisation error. 

We compute  two measurements of error in the mean,
\begin{align}\label{eq:error}  
\cE[ \| \phi - \phi^{h,N(h)} \|_{\infty}]\quad \text{ and } \quad \cE[
  | Q - Q_{h} |]
\end{align}  
(where $Q$ is defined in \eqref{eq:qoi} and we take $q = 1$). To
compute these, we  estimate $\phi$ by a reference solution with
$h^{-1} = 512$, $N = 256$, and we choose $3600$/ $2048$ KL  modes for
the cases $\nu = 0.5, 1.5$ respectively. The expectations in
\eqref{eq:error} are   estimated using a standard Monte Carlo
estimator (cf. \eqref{eq:estimatorMC}) with $32,768$ samples. Note
that when computing $\phi^{h, N(h)}(\omega, \cdot)$ we ignore the
theoretical path dependent stability criterion
\eqref{eq:hNlarge},  which led to the construction
\eqref{eq:stab_h_def}, and simply compute solutions with mesh
parameters $h$ and $N(h)$ for each sample.

Numerical computations of  \eqref{eq:error} are presented in  \Cref{fig:conv_rnd}.
For $\mathbb{E}[\Vert \phi - \phi^{h, N(h)} \Vert_\infty ]$ we
observe $\cO(h)$ convergence, for both  $\nu = 0.5, 1.5$,  even though
when  $\nu = 0.5$ we are only able to prove convergence of order
$\mathcal{O}(h^\eta)$ with $\eta < 0.5$.
We also observe smaller errors and a  faster  convergence rate (close to   $\cO(h^2)$)
for the error in the functional   $\cE[ | Q - Q_{h} |]$.

\begin{figure}[t]
\begin{center}
 \includegraphics[width=0.46\textwidth]{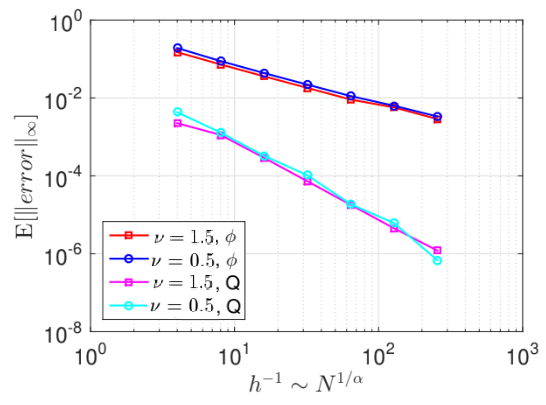}
 \caption{ Convergence of the mean error(s) $\cE[ \| \phi - \phi^{h,N(h)} \|_\infty]$ and
   $\cE[ | Q - Q_{h} |]$
       \label{fig:conv_rnd}}
\end{center}
\end{figure}

Our final  set of results concern the  $\epsilon$-cost  of the
standard (MC) and multilevel  (MLMC)  Monte Carlo methods
for computing $\mathbb{E}[Q]$ where $Q(\omega) = Q(\phi(\omega,
\cdot))$ and $Q$ is given by  \eqref{eq:qoi} with $q = 1$.   
We use  Method 2 of Example \ref{ex:solvers} as the linear system
solver for each realisation. Then Corollary \ref{cor:epscost} gives
theoretical  projections for the $\epsilon$-costs for each of these
methods in terms of $\eta$. The relevant values of $\eta$ are $\eta <
0.5$ when $\nu = 0.5$ and $\eta < 1$ when $\nu = 1.5$, in which case
$\chi>1$ (for $\nu = 0.5$) and $\chi >0$ (for $\nu = 1.5$).
Hence the theoretical $\epsilon$-costs given by Corollary \ref{cor:solver_cost} are
\begin{align}
 \label{epsobs} \mathcal{O}(\epsilon^{-s}) 
\end{align}
with $s$ as given in the column ``$s$, {\tt theory}'' in Table
\ref{table:MCrates}. To compare these to the observed
$\epsilon$-costs, we  give in the column ``$s$, {\tt observed}'' the
corresponding observed rates of growth of  $\epsilon$-cost when
estimating $\cE[Q]$, using the data which went into the construction
of Figure \ref{fig:epscost_hybrid}. 
 
The graphs in Figure \ref{fig:epscost_hybrid} depict the growth in
$\epsilon$-cost for each of the methods and each $\nu$ in the case
when $h = 1/512$ and $N = N(h)$ and show the superiority of the
multilevel method.  We observe from Table \ref{table:MCrates} and
Figure \ref{fig:epscost_hybrid} that for both values of $\nu$, the MLMC method
gives us excellent gains over the MC method in practice,
of at least one order of magnitude. The discrepancy between the
theory and numerics here arises because, for the two specific cases
considered, the observed value of $\alpha$ in  \eqref{eq:bias_assmc}
is somewhat higher than the theoretically predicted value of $\alpha =
\eta $ (with  $\eta < 0.5$ when $\nu = 0.5$ and $\eta < 1$ when $\nu = 1.5$).

\begin{figure}[t]
\begin{center}
\includegraphics[width=0.485\textwidth]{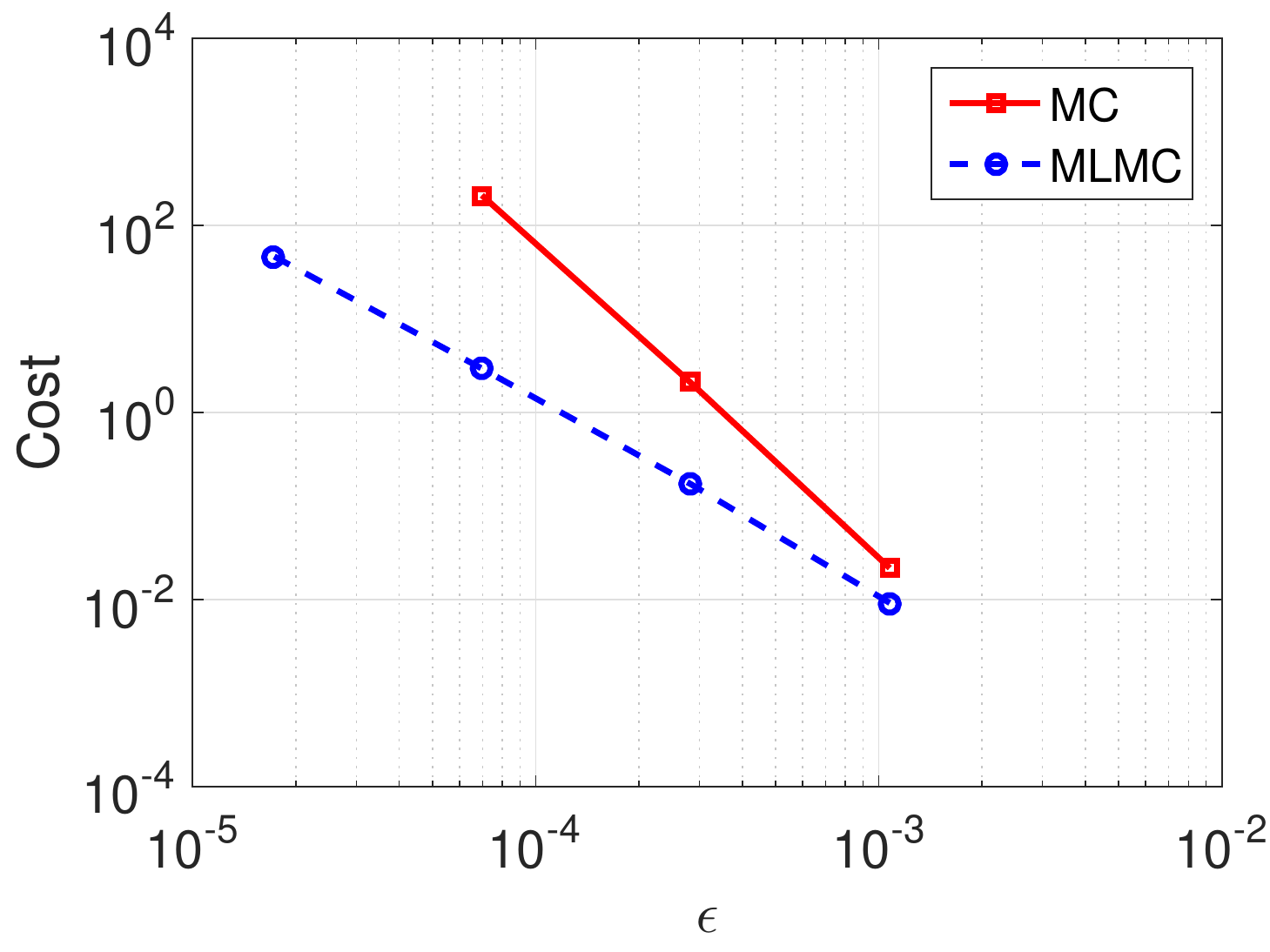}\hspace{2ex}\includegraphics[width=0.485\textwidth]{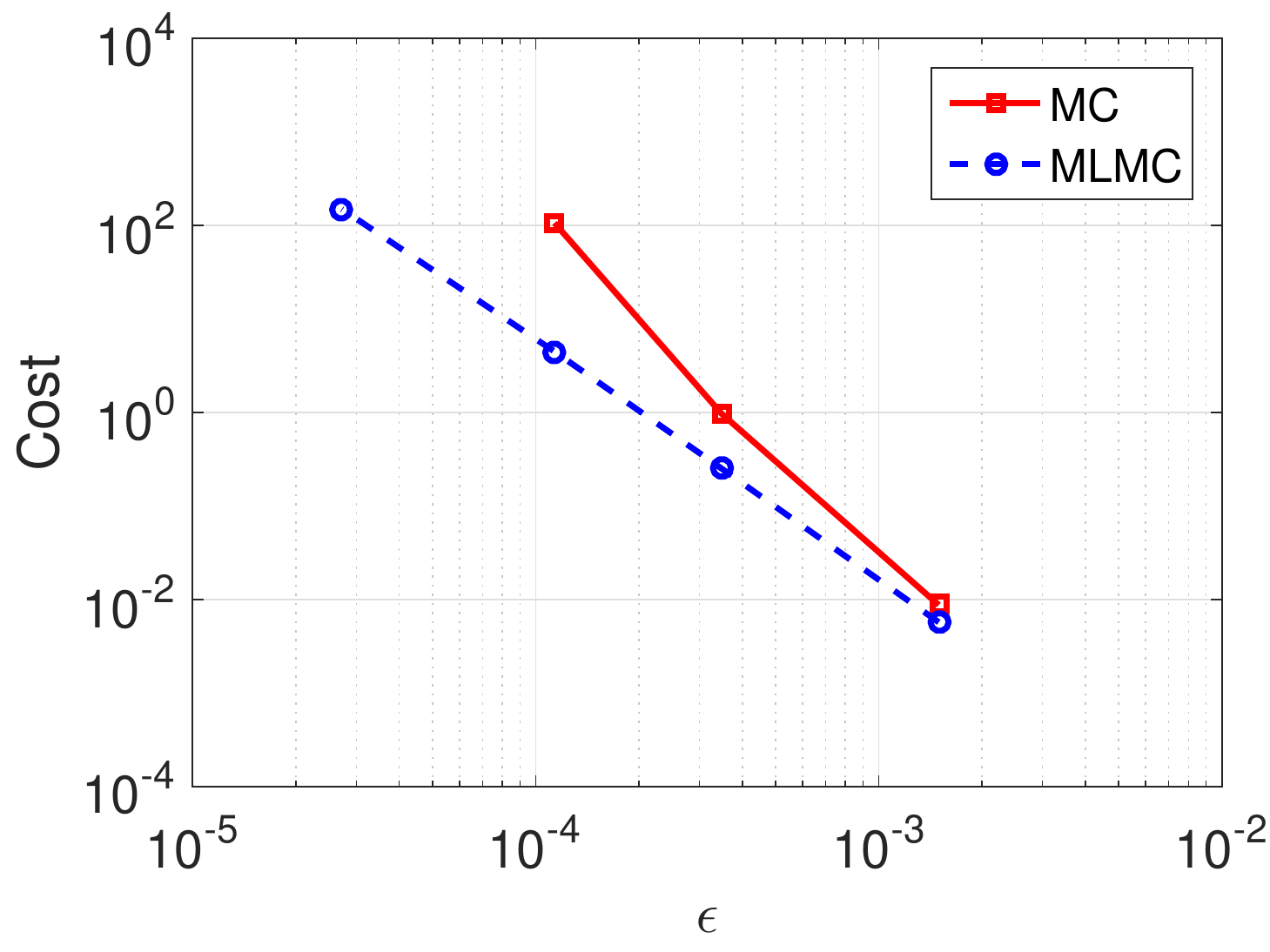}
\caption{Cost (in seconds) plotted against $\epsilon$  for the error (computed with respect to a reference solution)  for standard and multilevel Monte Carlo. (Left) $\nu = 1.5$ and (Right) $\nu = 0.5$.\label{fig:epscost_hybrid}}
\end{center}
\end{figure}

\begin{table}[t]
\centering
\begin{tabular}{|c|c|c|c|}
  \hline
  $\nu$ & Method  & $s$, {\tt theory}  & $s$,  {\tt observed} \\ \hline
  0.5 & MC & $> 5$ & $3.4$\\
  0.5 & MLMC & $>3$ & $2.4$ \\
  1.5 & MC & $>4$ & $3.3$ \\
  1.5 & MLMC & $>2$ & $2.1$\\
  \hline
 \end{tabular}
 \caption{Summary of computational $\epsilon$-cost rates with $s$ as in \eqref{epsobs}.\label{table:MCrates}} 
\end{table}

\section{Conclusion}

We have given a novel  error analysis  for  the discretised  heterogeneous transport equation, demonstrating
how the error depends on the heterogeneity. Although this is done for the 1d space and 1d  angle case (slab geometry) and  a classical discrete ordinates discretisation, the analysis could be extended to higher dimensional cases,
although with considerably more technicalities.  This analysis
is based on a careful analytical treatment of a certain underlying \igg{integral} equation.
We  then applied this analysis to the case when the cross-sections are  given by a random fields and presented 
error estimates in suitable Bochner norms for both the scalar flux and for functionals of it.
We assumed the input data could be piecewise continuous H{\"o}lder fields with low regularity.
We outlined   the Monte Carlo and multilevel Monte Carlo methods for quantifying the propagation of uncertainty
in this model problem. Using our probabilistic error estimates we then  rigorously proved estimates for the cost of these methods. These predict the superiority of the multilevel methods,  
even in the case of very rough input data.
Finally  we presented numerical results to support the theory.
Further numerical investigation of uncertainty quantification for the transport equation,  including some 2d in space and 1d in angle model problems is given in the PhD thesis \cite{Pa:18}.  

\bigskip 

\noindent 
{\bf Acknowledgement.}
Matthew Parkinson was supported by the
EPSRC Centre for Doctoral Training in Statistical Applied Mathematics at
Bath, 
under project EP/L015684/1.  We thank EPSRC and Wood plc.~for financial support
for this project and we particularly thank Professor Paul Smith of the Answers Software Team  for many helpful
discussions. This research made use of
the Balena High Performance Computing (HPC) Service
at the University of Bath.

\bibliographystyle{plain}

\end{document}